\RequirePackage{amsthm} 

\documentclass[sn-mathphys,Numbered]{sn-jnl}

\usepackage{graphicx}
\usepackage{multirow}
\usepackage{hyperref}
\usepackage{amsthm}
\usepackage{amssymb}
\usepackage{amsmath}
\usepackage{subcaption}
\usepackage{manyfoot}

\usepackage[dvipsnames]{xcolor}
\usepackage{color, colortbl}

\usepackage{booktabs}

\newcommand{\eps}{\varepsilon}
\newcommand{\scen}[1]{\mbox{\scriptsize{\rm{#1}}}}

\theoremstyle{thmstyleone}%
\newtheorem{theorem}{Theorem}
\newtheorem{proposition}[theorem]{Proposition}%
\newtheorem{corollary}[theorem]{Corollary}
\newtheorem{conjecture}[theorem]{Conjecture}

\theoremstyle{thmstyletwo}%
\newtheorem{example}[theorem]{Example}%

\theoremstyle{thmstylethree}%

\newcommand{\ph}{\phantom}

\newcommand{\smtxa}[2]{
{\mbox{\scriptsize
$\left[\!\! \begin{array}{#1} #2 \end{array} \!\! \right]$}}}

\newcommand{\bu}{\mathbf u}
\newcommand{\bv}{\mathbf v}
\newcommand{\bx}{\mathbf x}
\newcommand{\bxbar}{\overline\bx}

\newcommand{\bA}{\mathbf A}
\newcommand{\bB}{\mathbf B}
\newcommand{\bC}{\mathbf C}
\newcommand{\bI}{\mathbf I}
\newcommand{\bM}{\mathbf M}

\newcommand{\bS}{\mathbf S}

\newcommand{\bV}{\mathbf V}
\newcommand{\bZ}{\mathbf Z}
\newcommand{\bVfda}{\mathbf V_{\rm FDA}}
\newcommand{\bVtr}{\mathbf V_{\rm TR}}
\newcommand{\bVtrT}{\mathbf V_{\rm TR}^T}
\newcommand{\wtbVfda}{\widetilde{\mathbf V}_{\rm FDA}}
\newcommand{\wtbVtr}{\widetilde{\mathbf V}_{\rm TR}}
\newcommand{\bQ}{\mathbf Q}
\newcommand{\bW}{\mathbf W}
\newcommand{\bX}{\mathbf X}

\newcommand{\Deltabar}{\overline\Delta}
\newcommand{\bdelta}{\boldsymbol \delta}
\newcommand{\bmu}{{\boldsymbol \mu}} 
\newcommand{\bSigma}{\mathbf \Sigma}

\newcommand{\zero}{\mathbf 0}
\newcommand{\caln}{\mathcal N}
\newcommand{\R}{\mathbb R}
\newcommand{\wh}{\widehat}
\newcommand{\wt}{\widetilde}
\newcommand{\diag}{{\rm diag}}
\newcommand{\rank}{{\rm rank}}
\newcommand{\spooled}{{\mathbf S}_{{\rm pooled}}}

\newcommand{\cont}[1]{{}^c{#1}}
\newcommand{\tr}{{\rm tr}}
\DeclareMathOperator*{\argmin}{\arg\!\min}
\DeclareMathOperator*{\argmax}{\arg\!\max}
\DeclareMathOperator*{\var}{{\rm Var}}
\DeclareMathOperator*{\expec}{\mathbb{E}}

\usepackage[normalem]{ulem}

\begin{document}

\title{On the trace ratio method and Fisher's discriminant analysis for robust multigroup classification}

\author*[1]{\fnm{Giulia} \sur{Ferrandi}}\email{g.ferrandi@tue.nl}

\author[2]{\fnm{Igor V.} \sur{Kravchenko}}\email{igor.kravchenko@tecnico.ulisboa.pt}

\author[1]{\fnm{Michiel E.} \sur{Hochstenbach}}\email{m.e.hochstenbach@tue.nl}

\author[2]{\fnm{M.~Ros\'ario} \sur{Oliveira}}\email{rosario.oliveira@tecnico.ulisboa.pt}

\affil[1]{\orgdiv{Department of Mathematics and Computer Science}, \orgname{TU Eindhoven}, \orgaddress{\street{PO Box 513}, \city{Eindhoven}, \postcode{5600MB}, \country{The Netherlands}}}

\affil[2]{\orgdiv{CEMAT and Department of Mathematics}, \orgname{Instituto Superior T\'ecnico, Universidade de Lisboa}, \orgaddress{\street{Av.~Rovisco Pais}, \city{Lisbon}, \postcode{1049-001}, \country{Portugal}}}





\abstract{
We compare two different linear dimensionality reduction strategies for the multigroup classification problem: the trace ratio method and Fisher's discriminant analysis. Recently, trace ratio optimization has gained in popularity due to its computational efficiency, as well as the occasionally better classification results. However, a statistical understanding is still incomplete. We study and compare the properties of the two methods. Then, we propose a robust version of the trace ratio method, to handle the presence of outliers in the data. We reinterpret an asymptotic perturbation bound for the solution to the trace ratio, in a contamination setting. Finally, we compare the performance of the trace ratio method and Fisher's discriminant analysis on both synthetic and real datasets, using classical and robust estimators.
}



\maketitle



\section{Introduction}

Classification is an important statistical task where an observation is assigned to one of the non-overlapping known groups, based on the statistical properties of the data characterizing these groups. Statistics, machine learning, data science, and pattern recognition are some of the areas that use this family of methods to solve practical problems.
In this paper, we take a closer look at two linear methods for multigroup classification: the trace ratio method and Fisher's discriminant analysis.
By both theoretical analysis and numerical simulations, we attempt to increase our understanding of the properties of these methods.
We also introduce and study a robust version of the trace ratio method.

Suppose we have a population described by the random vector $\bX= (X_1, \dots,X_p)^T$, with mean $\expec(\bX) = \bmu$ and covariance $\var(\bX) = \bSigma$. We make the common assumption that $\bSigma$ is nonsingular.
The population is divided into $g$ groups (or classes), modeled by the class variable $Y \in \{1, \dots,g\}$. Denote the expectation of $\bX_j = (\bX\mid Y = j)$ by $\expec(\bX_j) = \bmu_j$, and its covariance matrix by $\var(\bX_j) = \bSigma_j$. The probability of being in group $j$ is $p_j = P(Y=j)$. Then
 \begin{equation}\label{eq:E(X)}
\bmu=\expec(\bX) = \expec(\expec(\bX \mid Y)) = \sum_{j=1}^g p_j \, \bmu_j.
\end{equation}
Given a new observation $\bx \in \mathbb{R}^{p}$, we would like to assign it to one of the $g$ classes, via a so-called \emph{classification rule}, that is built on top of the available data. We are interested in methods where the classification task includes a \emph{linear} dimensionality reduction step, where the data are projected onto a lower-dimensional subspace.

In particular, we study the well-known Fisher's discriminant analysis (FDA; see, e.g., \cite[Ch.~11]{johnson2007applied}) and a more recent dimensionality reduction method, the trace ratio method (TR; see, e.g., \cite{ngo2012trace}). Both strategies aim at finding a lower-dimensional representation of the data where the separation between the groups is maximized in a certain sense. They both involve a measure of distance between and within the groups. This is inspired by the decomposition of the covariance matrix $\var(\bX)$ as
\begin{equation}\label{eq:Var(X)}
\var(\bX) = \expec(\var(\bX\mid Y)) + \var(\expec(\bX\mid Y))
= \sum_{j=1}^g \, p_j\,\bSigma_j + \sum_{j=1}^g \, p_j \, (\bmu_j-\bmu) \, (\bmu_j-\bmu)^T,
\end{equation}
where $\bW=\sum_{j=1}^g \, p_j\,\bSigma_j$ is the \emph{within covariance} matrix and $\bB=\sum_{j=1}^g \, p_j \, (\bmu_j-\bmu) \, (\bmu_j-\bmu)^T$ the \emph{between covariance} matrix. 
If projecting the data leads to a good separation among the groups, a favorable effect on the performance of the classifiers is expected, since the present noise may be damped or filtered.
In addition, the dimension reduction step tends to increase the efficiency of subsequent steps and may help the analyst to interpret the classification output.

The within and between covariance matrices can be estimated in different ways. We start our discussion from the classical estimates in Section~\ref{sec:ldatr}, reviewing FDA and the trace ratio method. Next, we are interested in assessing the classification performances of the two methods when robust estimates are used. These are very useful when the data is contaminated, i.e., there is a certain proportion $\eps > 0$ of outliers that do not follow the distribution of $\bX$. As we will observe in Section~\ref{sec:simulexp}, while some nonlinear classifiers seem to be less affected by outliers, linear classifiers may benefit from the use of robust statistics. 

FDA with robust estimates has already been studied mainly by exploring the connection between FDA and linear discriminant analysis (LDA); see, e.g., \cite{Todorov.Pires:2007}. 
Several methods to robustify discriminant analysis have been proposed. The most straightforward proposals suggest employing robust estimators of group means and robust covariance matrices \cite{Todorov.Pires:2007,Hawkins.McLachlan:1997,Croux.Dehon:2001, Hubert.VanDriessen:2004,Bashir.Carter:2005}. 
These methods assume that the number of observations per group, compared to the number of variables, is sufficiently large to guarantee good estimates of the unknown parameters. When we are in the opposite situation (which happens in many practical problems), we say that our data is high-dimensional.
The first attempts to deal with this issue were done by Vanden Branden and Hubert
\cite{VandenBranden.Hubert:2005}, who provide a robust classifier for high-dimensional data based on SIMCA \citep{Wold:1976}, and by Pires and Branco \cite{Pires.Branco:2010} who propose robust classifiers based on projection-pursuit. Later on, Hoffmann and coauthors \cite{Hoffmann.et.al:2016} suggest a sparse and robust classification method based on partial least squares for binary classification problems, while \cite{ortner2020robust} presents a robust version of optimal scoring, for both the non-sparse and sparse case.

In this paper, we propose a new robust TR method, obtained by exploiting MCD robust estimates of the within and between covariance matrices. The method can deal with high-dimensional data since it uses regularized MCD estimates (MRCD, \cite{boudt2020mrcd}); however, as we will see in Section~\ref{SubSec:Discussion}, when a high number of irrelevant variables has a negative impact on the performance of the estimation methods, even on the robust ones. 

We derive theoretical expressions showing that the between and within theoretical covariance matrices, under a certain contamination scenario (cf.~Section~\ref{subsec:constscatter}), can be written as the sum of their non-contaminated counterparts and a perturbation term. In the TR case, this formulation allows us to derive an upper bound for the distance between the non-contaminated and the contaminated solutions of the method. 

We compare TR and FDA on synthetic and real datasets. Synthetic scenarios consider cases where one method performs better than the others. In this case, FDA and TR are used as classifiers and are compared by means of two different criteria. While the first one is based on the performance of the associated classification rules, the second criterion is related to the proximity between the true solution of one method and the estimated one. Real datasets have been chosen from the UCI \cite{ucirepo} and KEEL \cite{keelrepo} platforms, and they illustrate the performance of FDA and TR as dimensionality reduction methods, used before the construction of a classifier. In many of these datasets, FDA is as good as or better than TR. Moreover, robust TR shows clear improvements compared to classical TR in several datasets.

This paper is organized as follows. In Section~\ref{sec:ldatr} we review FDA, LDA, and the trace ratio method.
In Section~\ref{sec:Contam effect on W and B} we derive the theoretical within and between covariance matrices under a specific contamination scenario. 
For TR we derive an upper bound quantifying the effect of an infinitesimal percentage of contamination into the TR matrix of coefficients. 
In Section~\ref{sec:rob} we propose a robust TR estimator. In Section~\ref{sec:simulation} we carry out several simulation experiments and in Section \ref{sec:realdata} we explore FDA and TR as dimensionality reduction strategies for real datasets. The main conclusions are presented in Section \ref{sec:conclusions}.

We conclude this section by introducing some notations. We denote the $k \times k$ identity matrix by $\bI_k$.
While $\|\cdot\|$ denotes the standard Euclidean norm for vectors and the induced spectral norm for matrices (equal to its largest singular value), $\|\bx\|_{\bA}^2 = \bx^T\!\bA\bx$ is used for the $\bA$-weighted norm for a symmetric positive definite (SPD) matrix $\bA$.
Given the orthogonal bases of two subspaces (of possibly unequal dimension) $\bV$ and $\wh\bV$, the (largest) angle between the two subspaces is defined as
$\angle (\bV, \wh\bV) = \arcsin \|\bV\bV^T-\wh\bV\wh\bV^T\|$.
This quantity is well defined since $\|\bV\bV^T-\wh\bV\wh\bV^T\| \in [0,\,1]$ (see, e.g., \cite{knyazev2002principal}).

\section{FDA, LDA, and TR for multigroup classification}
\label{sec:ldatr}
We present and compare FDA \cite{fisher1936use} with the more recent trace ratio optimization method (see, e.g., \cite{ngo2012trace} for a review). 
We discuss the two methods for the case where classical estimates for the covariance matrices $\bB$ and $\bW$ are used.

Let $\bx_{j1}, \dots, \bx_{jn_j}$ represent $n_j$ realizations of a random sample from $\bX_j$; the total sample size is $n = \sum_{j=1}^g n_j $. The sample mean and the sample covariance matrix of group $j$ are $\bxbar_j = n_j^{-1} \, \sum_{h=1}^{n_j}{\bx_{jh}}$ and $\bS_j = (n_j-1)^{-1} \, \sum_{h=1}^{n_j} \, (\bx_{jh}-\bxbar_j) \, (\bx_{jh}-\bxbar_j)^T$, respectively.
The overall sample mean and covariance are then estimated by
$\bxbar = \sum_{j=1}^g \tfrac{n_j}{n} \, \bxbar_j$ and $\spooled = \sum_{j=1}^g \tfrac{n_j-1}{n-g} \, \bS_j$. 
We define the classical between-class scatter matrix, $\bB$, and the classical within-class scatter matrix, $\bW$, as follows:
\begin{equation} \label{eq:SbSw}
\bB = n^{-1} \, \sum_{j=1}^g n_j \, (\bxbar_j-\bxbar) \, (\bxbar_j-\bxbar)^T, \qquad
\bW = n^{-1} \, \sum_{j=1}^g \sum_{h=1}^{n_j} \ (\bx_{jh}-\bxbar_j) \, (\bx_{jh}-\bxbar_j)^T.
\end{equation}
To keep the notation simple, we use $\bB$ ($\bW)$ for both the theoretical and classical (i.e., non-robust) estimated between (within) covariance matrices. It is easy to check that $\bW = n^{-1} \, (n-g)\, \spooled$. 
Let $\rank(\cdot)$ indicate the rank of a matrix. Then $r = \rank(\bB) \le \min\{g-1, \, p\}$ and $\rank(\bW)\le \min\{n-g, \, p\}$. Throughout the paper, we assume the common situation that there are sufficiently many data points such that $n-g \ge p$. Since $\bSigma$ is assumed to be nonsingular, we also assume that $\bW$ is also nonsingular, and thus SPD. However, if some variables are highly correlated, numerical singularity of $\bW$ may occur. One possible way to handle this situation is described in Section~\ref{sec:realdata}.

\subsection{Fisher's discriminant analysis}\label{SubSec:FDA}
The original idea of Fisher \cite{fisher1936use} is to find the linear combination of the data, $\bv^T\bx$, that leads to the highest separability between the groups. By assuming that the population group covariance matrices are equal and of full rank, we look for a projection that maximizes the separation between the groups and minimizes the variability within the groups. This corresponds to finding a projection vector, $\bv_{\rm FDA}$, that maximizes a separability index: $\bv_{\rm FDA} = \argmax_{\bv \ne \zero}\rho_{\rm FDA}(\bv)$, where
\begin{equation} \label{eq:FDA.MaxProblem}
\rho_{\rm FDA}(\bv) = \frac{\sum_{j=1}^g n_j \, (\bv^T\bxbar_j-\bv^T\bxbar)^2}{\sum_{j=1}^g \sum_{h=1}^{n_j} (\bv^T\bx_{jh}-\bv^T\bxbar_j)^2} = \frac{\bv^T\bB\bv}{\bv^T\bW \bv}.
\end{equation}
This quantity is a {\em Rayleigh quotient} of the pair $(\bB, \bW)$.
Suppose we want to project our data onto a $k$-dimensional subspace spanned by the columns of $\bV\in\R^{p\times k}$, with $k \le r$.
FDA \eqref{eq:FDA.MaxProblem} can be formalized as a generalized eigenvalue problem; see, e.g., \cite[p.~623]{johnson2007applied}.

\begin{proposition}\label{prop:gep}
Let $\lambda_1 \ge \dots \ge \lambda_k>0$ be the $k\le r$ largest nonzero eigenvalues of the generalized eigenvalue problem $\bB\bv = \lambda\, \bW \bv$ and let $\bv_1, \dots, \bv_k$ be the associated eigenvectors (usually scaled such that $\bv_i^T\spooled\bv_i=1$). Then the vector of coefficients $\bv_{\rm FDA}$ that maximizes \eqref{eq:FDA.MaxProblem} is given by $\bv_{\rm FDA} = \bv_1$. The $i$th eigenvector is the solution to
\begin{equation}
\label{eq:gep1}
\bv_{i} = \argmax_{\substack{\bv_j^T\spooled \bv = 0,\ j < i \\ \bv \ne \zero}} \ \frac{\bv^T\bB\bv}{\bv^T\bW \bv}.
\end{equation}
\end{proposition}
The linear combination $\bv_1^T\bx$ is called \emph{sample first discriminant}, and likewise the linear combination $\bv_i^T\bx$, for $i \le k$, is the sample $i$th discriminant; Fisher's discriminants generate a lower-dimensional space whose span optimizes a sequence of separation criteria $\rho_{\rm FDA}(\cdot)$ defined in \eqref{eq:FDA.MaxProblem}.

Problem \eqref{eq:gep1} is equivalent to finding the first $k$ generalized eigenvectors $\bVfda = [\bv_1, \dots, \bv_k]$ of the pencil $(\bB, \bW)$, with scaling $\bVfda^T\spooled\bVfda = \bI_k$. Since $\bW$ and $\spooled$ are the same matrix up to a multiplicative factor, the generalized eigenvalue problem for FDA is equivalent to finding the solution $\bVfda$ to the constrained optimization problem
(cf., e.g., \cite{ngo2012trace})
\begin{equation} \label{eq:gep}
\max_{\bV^T\bW \, \bV = \bI_k} \tr(\bV^T\bB\bV),
\end{equation}
where ${\rm tr}(\cdot)$ indicates the trace of a matrix. Another equivalent formulation is (cf.~\cite[Prop.~2.1.1]{absil2009optimization}) 
$\max_{\text{\scriptsize rank}(\bV) = k} {\rm tr}((\bV^T\bW \, \bV)^{-1} \, \bV^T\bB\bV)$,
where the objective function may be viewed as a \emph{generalized Rayleigh quotient}.

\subsection{Relation between linear and Fisher's discriminant analysis}\label{sec:fda-lda}

When we assume that all the group covariance matrices are equal, i.e., $\bSigma_j \equiv \bSigma$, FDA has a strong connection with LDA, which assumes that, within groups, the random variables have a multivariate normal distribution with equal covariance matrices. The associated classification rule is derived such that it minimizes the \emph{total probability of misclassification}, which is the probability of assigning a given data point to an incorrect class (cf.~\cite[Result~11.5]{johnson2007applied}). According to LDA, an observation $\bx$ is assigned to class $j$ if (see, e.g., \cite[p.~630]{johnson2007applied})
\begin{equation} \label{eq:lda}
j = \argmin_i \|\bx-\bmu_i\|_{\bSigma^{-1}}^2-2\, \log p_i,
\end{equation}
and $\|\bx-\bmu_i\|_{\bSigma^{-1}}$ is the Mahalanobis distance between $\bx$ and the group mean $\bmu_i$.
LDA is also the classifier that maximizes the total probability of correctly assigning an observation (called the accuracy) among all possible classifiers for multivariate normal mixtures with equal covariance matrices. 


When $k = r$, LDA with the classical estimates $\bxbar_i$ and $\spooled$ coincides with projecting a data point $\bx$ onto $\bVfda \in \mathbb R^{p\times r}$ and classifying it to the group with the closest projected mean, according to the Euclidean distance.
This result is reported in Proposition \ref{prop:eulda}, which is a slightly different formulation of \cite[Result~11.6]{johnson2007applied}. Recall that $\bB$ has $r\le \min{\{g-1,p\}}$ positive eigenvalues.

\begin{proposition}
\label{prop:eulda}
Let $[\bVfda\; \bV_\perp]$ contain the $\spooled$-orthogonal generalized eigenvectors of the pencil $(\bB,\bW)$, with $\bVfda = [\bv_1,\dots,\bv_r]$ and $\bV_\perp^T\spooled\bVfda = \zero$. Then only the terms $\|\bVfda^T(\bx - \bxbar_i)\|^2$ contribute to the classification. In mathematical terms,
\[
\argmin_i \|\bx-\bxbar_i\|_{\spooled^{-1}}^2-2\, \log \tfrac{n_i}{n} = \argmin_i \|\bVfda^T (\bx-\bxbar_i)\|^2-2\, \log \tfrac{n_i}{n}.
\]
\end{proposition}
\begin{proof}
Since $\rank(\bB) = r$, we have $\bB\bV_\perp = \zero$. In particular, $\bV_\perp^T(\bxbar_i - \bxbar) = \bf 0$ for all $i$. This can be deduced from the fact that $0 = \text{tr}(\bV^T_\perp\bB\bV_\perp) = \sum_{i=1}^g \frac{n_i}{n} \, \|\bV^T_\perp(\bxbar_i-\bxbar)\|^2$.
Since $\spooled^{-1} = \bVfda\bVfda^T + \bV_\perp\bV_\perp^T$, we have
\begin{align*}
\|\bx-\bxbar_i\|_{\spooled^{-1}}^2 & = \|\bVfda^T (\bx-\bxbar_i)\|^2 + \|\bV_\perp^T \,[\, (\bx-\bxbar) + (\bxbar-\bxbar_i)\,]\,\|^2\\
& = \|\bVfda^T (\bx-\bxbar_i)\|^2 + \|\bV_\perp^T(\bx-\bxbar)\|^2,
\end{align*}
where the second term in the right-hand side is common to all groups.
\end{proof}

From this last proposition, it is easy to see that Rule \eqref{eq:lda} is equivalent to classifying a data point $\bx$ to the class with the closest projected mean $\bVfda^T \bxbar_j$, up to an additive weight related to the class proportions.

For a reduced dimension $k < r$,
there is still a connection between Rule \eqref{eq:lda} and the classification via the Euclidean distance. This requires an estimation of the group means and the common covariance matrix by assuming that the group means lie in a $k$-dimensional space. The result is the following.

\begin{proposition} \cite[Prop.~1]{hastie1996reducedLDA} \label{prop:eulowrank}
Consider a family of $g$ multivariate normal distributions, with mean vectors $\bmu_i$ and common covariance matrix $\bSigma$. Let $\bVfda$ be the $p\times k$ solution to the generalized eigenvalue problem $(\bB, \bW)$ \eqref{eq:gep}, scaled such that $\bVfda^T\spooled\bVfda = \bI_k$. Let $\wh \bmu_i$ and $\wh \bSigma$ be the maximum likelihood estimators of $\bmu_i$ and $\bSigma$, with the additional constraint that the centroids $\bmu_i$ lie in a $k$-dimensional space. Then
$\wh \bmu_i = \spooled\bVfda\bVfda^T\,\bxbar_i + (\bI_p - \spooled\bVfda\bVfda^T) \, \bxbar$ and
$\wh \bSigma = \spooled + (\bI_p - \spooled\bVfda\bVfda^T) \, \bB \, (\bI_p - \spooled\bVfda\bVfda^T)^T.$
Moreover,
\[
\|\bx-\wh\bmu_i\|_{\wh\bSigma^{-1}}^2 = \|\bVfda^T \bx-\bVfda^T\bxbar_i\|^2 + (\bx - \bxbar)^T\,(\wh\bSigma^{-1} - \bVfda\bVfda^T)\,(\bx - \bxbar),
\]
where the rightmost term is common to all groups.
\end{proposition}

We can conclude that classifying a projected data point according to the Euclidean distance from the projected means is equivalent to classifying it in the original space according to the Mahalanobis distance, where $\bmu_i$ and $\bSigma$ are estimated appropriately based on the rank of $\bVfda$, i.e., by $\bxbar_i$ and $\spooled$ when $\rank(\bVfda) = r$ (Proposition~\ref{prop:eulda}), and by $\wh\bmu_i$ and $\wh\bSigma$ when $\rank(\bVfda) = k < r$ (Proposition~\ref{prop:eulowrank}). 
In other words, regardless of the rank of $\bVfda$, $\bx$ is assigned to class $j$ if
\begin{equation}
\label{eq:fda-classifier}
j = \argmin_i \|\bVfda^T (\bx-\bxbar_i)\|^2-2\, \log \tfrac{n_i}{n},
\end{equation}
where $\frac{n_i}{n}$ estimates the prior probability $p_i$ of group $i$. 

We remark that the classification rule \eqref{eq:fda-classifier} does not require any distributional assumption. In view of this, in general, it is not a Bayesian classification rule and thus does not minimize the total probability of misclassification. Nevertheless, FDA with $k = r$ is optimal under the homoscedastic assumption and with multivariate normally distributed data; in this case FDA and LDA coincide.

\subsection{Trace ratio optimization}\label{Sec:TR optimization}
Recently, a different alternative to FDA has been studied (see, e.g., \cite[Ch.~10]{Fukunaga:1990} and \cite{Park:2008,ngo2012trace,shi2021new}).
The essence is that instead of the trace of a ratio (see Section~\ref{SubSec:FDA}), a {\em ratio of traces} is considered.
If we consider the projected data $\bV^T\bx$, another possible generalization of $\rho_{\rm FDA}$ \eqref{eq:FDA.MaxProblem} comes from considering the Euclidean distance between the projected vectors. The between covariance matrix of the projected data can be written as a function of the distances between $\bV^T\bxbar_j$ and $\bV^T\bxbar$ (see, e.g., \cite{ngo2012trace}) in the following way:
\begin{eqnarray*}
\sum_{j=1}^g n_j \, \| \bV^T\bxbar_j-\bV^T\bxbar\|^2
& = & \sum_{j=1}^g n_j \, (\bV^T\bxbar_j-\bV^T\bxbar)^T \, (\bV^T\bxbar_j-\bV^T\bxbar) \\
& = & \sum_{j=1}^g \tr (n_j \, \bV^T(\bxbar_j-\bxbar) \, (\bxbar_j-\bxbar)^T \, \bV) \ = \ n\ \tr (\bV^T\bB\bV).
\end{eqnarray*}
Similarly, for the common variability within the projected groups we have
\[
\sum_{j=1}^g \, \sum_{h=1}^{n_j} \, \| \bV^T\bx_{jh}-\bV^T\bxbar_j\|^2 = n\ \tr (\bV^T\bW \, \bV).
\]
The new optimization problem aims at finding the orthogonal projection $\bVtr \in \mathbb R^{p\times k}$ that maximizes the trace ratio (TR):
\begin{equation}\label{eq:TR.MaxProblem}
\bVtr = \argmax_{\bV^T\bV = \bI_k} \ \rho(\bV) = \frac{\tr(\bV^T\bB\bV)}{\tr(\bV^T\bW \, \bV)}.
\end{equation}
Other variants of the problem \eqref{eq:TR.MaxProblem} have been considered in \cite[Ch.~10]{Fukunaga:1990} and \cite{Park:2008,ngo2012trace,shi2021new}; these include a different set of constraints on $\bV$ or consider matrices different from $\bW$ in the denominator of $\rho$. 

In the one-dimensional case, when $k = 1$, problems \eqref{eq:FDA.MaxProblem}, \eqref{eq:gep}, and \eqref{eq:TR.MaxProblem} coincide, up to the scaling of the solution vector. For $k > 1$, the trace ratio problem \eqref{eq:TR.MaxProblem} does not have an explicit solution and thus relies on iterative approaches; see, e.g., \citep{Jia.et.al:2009,ngo2012trace,Wang.et.al:2007,Zhao.et.al:2013}. Still, the authors of \cite{shi2021new} report that the TR problem reduces to a least squares problem for the high-dimensional case
$p \ge n$, and provide a closed form solution. 

The existence of $\bVtr$ is guaranteed under a milder hypothesis on the within covariance matrix. We hereby report \cite[Prop.~3.2]{ngo2012trace} for completeness.

\begin{proposition}\label{prop:rank}
Let $\bB$ be symmetric and $\bW$ symmetric positive semidefinite, with $\rank(\bW) \ge p-k+1$. Then \eqref{eq:TR.MaxProblem} admits a finite maximum.
\end{proposition}

In particular, Proposition~\ref{prop:rank} implies that for a finite maximum, $k$ may be larger than $r = \rank(\bB) \le \min\{g-1, p\}$, and $\bW$ does not need to be positive definite.
Several authors (see, e.g., \cite{ngo2012trace}) mention the constraint $\bV^T \bC \bV = \bI_k$ for a certain SPD matrix $\bC$.
Since the eigenvectors of the pair $(\bB, \bW)$ are both $\bB$-orthogonal and $\bW$-orthogonal, we remark that for the convex combination $\bC = (1-\gamma_1)\,\bB + \gamma_1 \bW$, with $0 < \gamma_1 \le 1$ such that $\bC$ is SPD, the TR method is equivalent to a generalized eigenvalue problem \eqref{eq:gep}, which renders FDA.

Algorithm~1 shows the pseudocode for a natural iterative method (see, e.g., \cite{ngo2012trace}) to find the solution to \eqref{eq:TR.MaxProblem}.

\noindent\vrule height 0pt depth 0.5pt width \textwidth \\[-1mm]
{\bf Algorithm~1: Iterative method to solve the trace ratio problem \eqref{eq:TR.MaxProblem}} \\[-3.5mm]
\vrule height 0pt depth 0.3pt width \textwidth \\
{\bf Input:} $\bB$, $\bW \in \R^{p \times p}$, dimension $k$, initial $(p \times k)$ matrix $\bV$ with orthogonal columns. \\
{\bf Output:} Solution $\bV\in \R^{p \times k}$ to \eqref{eq:TR.MaxProblem}.\\
\begin{tabular}{ll}
{\footnotesize 1:} & If $\bV$ is not provided, choose a random $p \times k$ matrix with orthogonal columns \\
& {\bf while not} converged (as determined by $\rho(\bV)$ or $\bV$) \\
{\footnotesize 2:} & \ph{M} $\rho = \rho(\bV) = \frac{\tr(\bV^T\bB\bV)}{\tr(\bV^T\bW \, \bV)}$ \\
{\footnotesize 3:} & \ph{M} Determine (orthogonal) $\bV$ as largest $k$ eigenvectors of $\bB-\rho \bW$ \\
& {\bf end} \\
\end{tabular} \\
\vrule height 0pt depth 0.5pt width \textwidth

We would like to point out that we may also start Algorithm~1 with the orthogonalized FDA solution $\bV = \bVfda \in \mathbb R^{p\times k}$.

The global convergence of Algorithm~1 is stated in, e.g., \cite[Thm.~5.1]{zhang2010fast}. There the uniqueness of the solution $\bVtr$ is also discussed, and it is related to line 3 of Algorithm~1: if there is a gap between the largest $k$ eigenvalues of $\bB - \rho(\bVtr)\,\bW$ and the $p-k$ smallest ones, then the solution is uniquely described by the set $\{\bVtr\bQ\, : \, \bQ \in \mathbb{R}^{k\times k},\, \bQ^T \bQ = \bI_k\}$.

The stopping criterion may be based on the (relative) error between two consecutive $\rho$-values, or the angle between two consecutive subspaces spanned by the columns of $\bV$; in our implementation, we use the first approach.

Unlike FDA, TR does not come with a natural classifier. In \cite{ngo2012trace}, the method of the $k$-nearest neighbors is used to classify new data points in the projected space. We propose to classify a data point to the group with the closest projected sample mean vectors, according to the Euclidean distance and the class proportions:
\[
j = \argmin_i \|\bVtr^T (\bx-\bxbar_i)\|^2-2\, \log\tfrac{n_i}{n},
\]
as it is done in FDA, although we do not have a connection with the Mahalanobis distance in the original space and the Euclidean distance in the projected space.

\subsection{FDA, TR, and the reduced dimension}
We now study the solution to FDA and TR, and some relevant quantities, as a function of $k$, the selected reduced dimension.
We recall that for FDA, $1 \le k \le r = \rank(\bB)$; since $\bW$ is assumed nonsingular, the reduced dimension for TR is not restricted, thus $1 \le k \le p = \rank(\bW)$.

The FDA solution $\bVfda^{(k)}$ corresponds to the first $k$ generalized eigenvectors of $(\bB,\bW)$; therefore it holds $\text{span}(\bVfda^{(k)}) \subset \text{span}(\bVfda^{(k+1)})$ for $k = 1, \dots, r-1$. If we are interested in choosing the best value for $k$, it is enough to compute the eigenvalue decomposition of $(\bB,\bW)$ once and retain the desired number of eigenvectors. Another consequence is the following result. For simplicity, we choose $\bVfda^{(k)}$ to have $\bW$-orthogonal eigenvectors (instead of $\spooled$-orthogonal).

\begin{proposition}
Let the eigenvectors of $(\bB, \bW)$ be scaled $\bW$-orthogonal. Then
$\tr((\bVfda^{(k)})^T \, \bB\bVfda^{(k)}) = \lambda_1 + \cdots + \lambda_k$, \ $\tr((\bVfda^{(k)})^T \, \bW \, \bVfda^{(k)}) = \tr(\bI_k)=k$, and therefore both quantities are increasing as a function of $k$.

The trace ratio $\rho(\bVfda^{(k)})$ for the FDA subspace is equal to $(\lambda_1 + \cdots + \lambda_k)/k$, and therefore is a non-increasing function of $k$. 
\end{proposition}

The quantity $\tr((\bVfda^{(k)})^T \, \bB\bVfda^{(k)})$ is mentioned in \cite[(11-68)]{johnson2007applied} as a theoretical criterion to choose $k$.

Now we turn our attention to the TR problem.
Algorithm~1 needs the value of $k$ in advance; the processes and outputs for different values of $k$ are generally unrelated, and therefore we expect $\text{span}(\bVtr^{(k)}) \not\subset \text{span}(\bVtr^{(k+1)})$.
The following example illustrates this statement.
\begin{example} \rm
By applying an orthogonal transformation, one may choose the within scatter matrix $\bW$ to be diagonal without loss of generality.
Therefore, we generate $1000$ random problems, where $\bW = {\rm diag}(1,2, \dots, 20)$ is fixed, while the between scatter matrix is given by $\bB = \bA\bA^T$, where $\bA$ is a $20 \times 10$ matrix with normally distributed entries, after which $\bA$ is column centered. The columns represent $10$ groups.
For each problem, we run TR with $k = 1,2,\dots,8$ and compute the angle between two consecutive solutions, $\bVtr^{(k)}$ and $\bVtr^{(k+1)}$. Figure~\ref{fig:angles}(a) reports the sorted angles $\angle(\bVtr^{(2)}, \, \bVtr^{(3)})$ for each problem. Their values range from rather small to approximately $\pi/5$, meaning that the two subspaces can be rather different. 
In Figure~\ref{fig:angles}(b), we carry out similar experiments, but now we display the median $\angle(\bVtr^{(k)}, \, \bVtr^{(k+1)})$ over 1000 experiments for various values of $k$.

\begin{figure}[htb!] 
\centering
\includegraphics[width=0.3\linewidth]{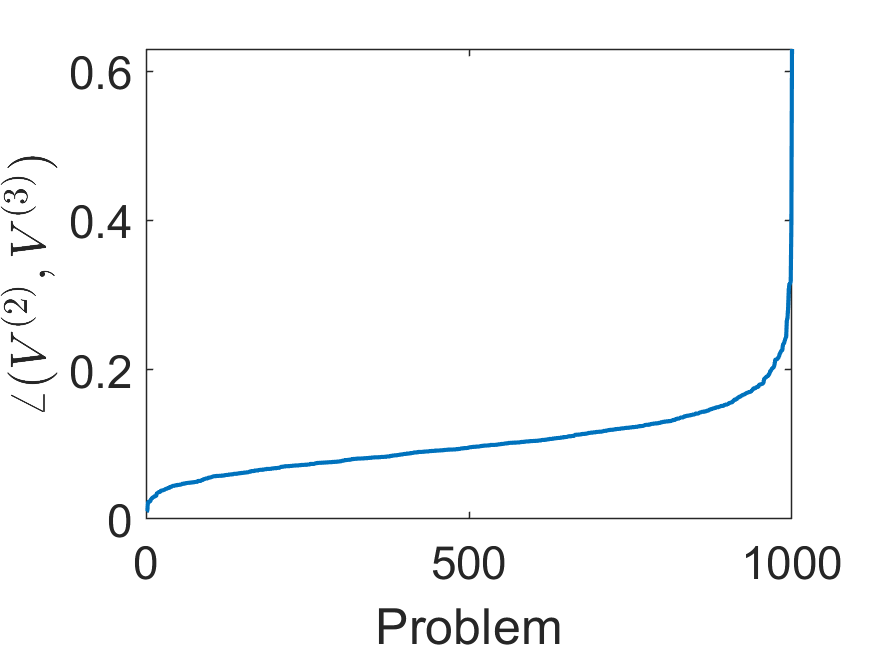} \qquad
\includegraphics[width=0.3\linewidth]{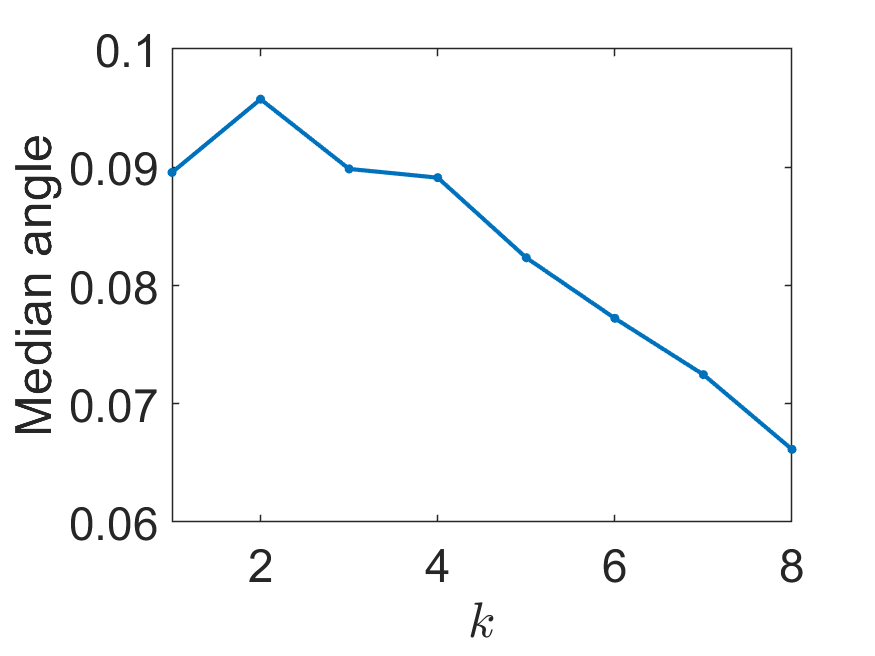} \\ (a) \hspace{40mm} (b)
\caption{(a): Subspace angle between two-dimensional and three-dimensional TR spaces $\bVtr^{(2)}$ and $\bVtr^{(3)}$. (b): Median angle  $\angle(\bVtr^{(k)}, \, \bVtr^{(k+1)})$ over 1000 random problems, for $k = 1, \dots, 8$.}
\label{fig:angles}
\end{figure}
\end{example}

Since $\bVtr^{(k)} \not\subset \bVtr^{(k+1)}$, the quantities $\tr((\bVtr^{(k)})^T \, \bB\bVtr^{(k)})$ and $\tr((\bVtr^{(k)})^T \, \bW\bVtr^{(k)})$ do not have an easy expression and interpretation. Nevertheless, we can show that the optimal trace ratio $\rho^{(k)} = \rho(\bVtr^{(k)})$ is a non-increasing function of $k$. To the best of our knowledge, this result is new.

\begin{proposition}\label{Decrease_of_rho}
Let $\rho^{(k)}$ be the optimal function value of the TR, for a $k$-dimensional subspace of $\mathbb R^p$ and the pencil $(\bB,\,\bW)$, where $\bW$ is SPD. Then $\rho^{(k)}$ is a non-increasing function of $k$, for $k = 1, \dots, p$.
\end{proposition}
\begin{proof}
Let $\bVtr^{(k)} \in \mathbb{R}^{p\times k}$ be the solution to TR with optimal value $\rho^{(k)}$, for $k = 1, \dots p-1$. Because of the linearity of the trace, we can write
\[
\rho^{(k+1)} - \rho^{(k)} = \frac{\tr((\bVtr^{(k+1)})^T \, (\bB - \rho^{(k)}\bW) \, \bVtr^{(k+1)})}{\tr((\bVtr^{(k+1)})^T \, \bW \bVtr^{(k+1)})},
\]
where the denominator is positive.
Denote by $\lambda_i(\bA)$ the $i$th largest eigenvalue of a matrix $\bA$.
We recall that
\[
0 = \tr((\bVtr^{(k)})^T \, (\bB - \rho^{(k)}\bW) \, \bVtr^{(k)}) = \sum_{i=1}^k \lambda_i(\bB - \rho^{(k)}\bW).
\]
The first equality comes from the first-order necessary condition for the maximum (see, e.g., \cite{ngo2012trace}), while the second follows from the fact that $\bVtr^{(k)}$ contains the $k$ largest eigenvectors of $\bB-\rho^{(k)} \bW$.
Since the sum of the $k$ largest eigenvalues of $\bB - \rho^{(k)}\bW$ is zero, the $(k+1)$st largest eigenvalue must be nonpositive: $\lambda_{k+1}(\bB - \rho^{(k)}\bW) \le 0$. Then 
\[
\tr((\bVtr^{(k+1)})^T(\bB - \rho^{(k)}\bW) \, \bVtr^{(k+1)}) \le
\max_{\bV^T\bV = \bI_{k+1}}
\tr(\bV^T(\bB - \rho^{(k)}\bW)\,\bV) = \lambda_{k+1}(\bB - \rho^{(k)}\bW)
 \le 0.
\]
This proves that $\rho^{(k+1)} - \rho^{(k)} \le 0$.
\end{proof}

The monotonicity from this proposition does not need to be strict, in the case of a multiple generalized eigenvalue.
An easy example of this is $\bB = \diag(2,4)$ and $\bW = \diag(1,2)$, when $\rho^{(1)} = \rho^{(2)} = 2$.
We also note that, for this result, the hypothesis that $\bW$ is SPD can be weakened; one only needs that $\tr((\bVtr^{(k)})^T \, \bW \bVtr^{(k)})>0$ for all $k$.

By rather extensive experimentation, we have reason to believe that $\tr((\bVtr^{(k)})^T\bB\bVtr^{(k)})$ is a non-decreasing function of $k$, just as for FDA. A proof or counterexample is currently lacking.

\begin{conjecture} \label{conj:1}
Let $\bB$ be symmetric positive semidefinite and $\bW$ be SPD, and let $\bVtr^{(k)}$ be the quantity determined by the trace ratio method.
The quantity $\tr((\bVtr^{(k)})^T\bB\bVtr^{(k)})$ is a non-decreasing function of $k$.
\end{conjecture}

If this conjecture is true, the following related result follows as well.

\begin{corollary}
If Conjecture~\ref{conj:1} holds, then $\tr((\bVtr^{(k)})^T \, \bW \bVtr^{(k)})$ is also non-decreasing in $k$.
\end{corollary}
\begin{proof}
This follows from Proposition \ref{Decrease_of_rho}.
In fact, since $\rho^{(k+1)} \le \rho^{(k)}$,
\[
\frac{\tr((\bVtr^{(k+1)})^T\bB\bVtr^{(k+1)})}{\tr((\bVtr^{(k)})^T\bB\bVtr^{(k)})} \le \frac{\tr((\bVtr^{(k+1)})^T \, \bW \bVtr^{(k+1)})}{\tr((\bVtr^{(k)})^T \, \bW \bVtr^{(k)})}.
\]
\end{proof}

Therefore, if Conjecture~\ref{conj:1} holds, since the trace ratio is decreasing, the quantity $\tr((\bVtr^{(k)})^T \, \bW \bVtr^{(k)})$ increases faster than $\tr((\bVtr^{(k)})^T\bB\bVtr^{(k)})$ as $k$ increases.


From the descriptions of FDA and TR methods, it is clear that any pencil $(\bB,\bW)$, with $\bB$ symmetric and $\bW$ SPD, estimating the between and the within covariance matrices, may be exploited in the contexts of FDA and TR. This motivates the use of robust estimators in Section~\ref{sec:rob} and of the true covariance matrices in the simulation study (Section~\ref{sec:simulation}).

Although FDA seems to have a stronger statistical motivation than TR, in Section~\ref{sec:simulation} we illustrate some situations in which TR performs better than FDA in the classification task.


\section{Effect of contamination in the within and between covariance matrices}\label{sec:Contam effect on W and B}


In the following sections, we express the contaminated theoretical within and between covariance matrices as the original matrices plus a perturbation term. These perturbation terms allow us to derive an upper limit of a proximity measure between the TR solution when no contamination is present, and their equivalent under contamination.

\subsection{Contaminated theoretical between and within covariance matrices}\label{subsec:constscatter}
We model the presence of atypical data as a mixture of two random vectors: the majority of the data comes from a certain (central, clean, or non-contaminated) mixture model $\bX$, and, in each group, there is a probability $\eps$ for getting data generated by a different distribution, characterizing the atypical scenario. Let $\cont\bX_j$ be the contaminating random variable, where $\expec(\cont\bX_j)=\cont\bmu_j$ and $\var(\cont\bX_j)=\cont\bSigma_j$. Let $\bZ_j$ be the contaminated variable that describes the distribution of group $j$:
\[
\bZ_j = (1-\Omega_j)\, \bX_j + \Omega_j \, \cont\bX_j,
\]
with $\Omega_j \sim {\rm Bernoulli}(\eps)$.
In addition, we assume that $\Omega_j$, $\bX_j$, and $\cont\bX_j$ are independent random variables for any value of $j$. Then, we can derive the mean vector and covariance matrix of $\bZ_j$:
\begin{align}
\expec(\bZ_j) & = (1-\eps)\, \bmu_j + \eps \ \cont\bmu_j = \bmu_j + \eps \, (\cont\bmu_j-\bmu_j) = \bmu_j + \eps \, \Delta_j, \quad {\rm where} \quad \Delta_j = \cont\bmu_j-\bmu_j. \label{eq:Bc_4}
\end{align}
The covariance matrix of $\bZ_j$ is given by
\begin{align}
\var(\bZ_j) & = \expec(\bZ_j\bZ_j^T) -\expec(\bZ_j) \expec(\bZ_j)^T\nonumber\\
& = \eps \, \expec(\cont\bX_j(\cont\bX_j)^T) + (1-\eps) \, \expec(\bX_j(\bX_j)^T) \nonumber\\ & \ph{MMM} -\eps^2 \ \cont\bmu_j \, (\cont\bmu_j)^T - (1-\eps)^2 \, \bmu_j \, \bmu_j^T - \eps \,(1-\eps)\,\left(\cont\bmu_j \, \bmu_j^T + \bmu_j \,(\cont\bmu_j)^T\right)\nonumber\\
& = \eps \ \cont\bSigma_j + (1-\eps) \, \bSigma_j + \eps \,(1-\eps)\,\left(\cont\bmu_j \,(\cont\bmu_j)^T + \bmu_j \, \bmu_j^T-\cont\bmu_j\bmu_j^T-\bmu_j(\cont\bmu_j)^T\right)\nonumber\\
& = \eps \ \cont\bSigma_j + (1-\eps) \, \bSigma_j +\eps \, (1-\eps) \, \Delta_j \, \Delta_j^T = \bSigma_j + \eps\left(\cont\bSigma_j-\bSigma_j\right) + \eps \, (1-\eps) \, \Delta_j \, \Delta_j^T .\label{eq:varz1}
\end{align}
Let $\bZ$ be the random vector describing the mixture of the contaminated random vectors $\bZ_j$.
From \eqref{eq:E(X)}, the expectation of $\bZ$ is 
\begin{equation}
\label{eq:Bc_5}
\expec(\bZ) = \sum_{j=1}^g p_j \expec(\bZ_j) 
= \bmu + \eps \, \Deltabar, \quad {\rm where} \quad \Deltabar = \sum_{j=1}^g p_j \, \Delta_j.
\end{equation}
From \eqref{eq:Var(X)} and 
\eqref{eq:varz1}, we can compute the contaminated theoretical within covariance matrix:
\begin{equation}\label{eq:W_contam}
\bW_Z = \sum_{j=1}^g p_j\var(\bZ_j) = \bW + \eps \,\sum_{j=1}^g p_j\,(\cont\bSigma_j-\bSigma_j) + \eps \, (1-\eps) \, \sum_{j=1}^g p_j\,\Delta_j \, \Delta_j^T,
\end{equation}
where $\bW = \sum_{j=1}^g p_j \, \bSigma_j$ is the theoretical within covariance matrix. We notice that when $\cont\bSigma_j = \bSigma_j$ for all $j$, the contaminated theoretical within covariance matrix is the original $\bW$, plus a low-rank matrix.

By using \eqref{eq:Var(X)}, we compute the theoretical between covariance matrix of $\bZ$ by substituting \eqref{eq:Bc_4} and \eqref{eq:Bc_5}:
\begin{align*}
\bB_Z &= \sum_{j=1}^g p_j \, (\expec(\bZ_j) - \expec(\bZ)) \, (\expec(\bZ_j) - \expec(\bZ))^T \\[-1mm]
&= \sum_{j=1}^g p_j \, (\bmu_j +\eps\Delta_j - \bmu - \eps \, \Deltabar) \, (\bmu_j +\eps\Delta_j - \bmu - \eps \, \Deltabar)^T.
\end{align*}
Let us define $\bdelta_j = \bmu_j-\bmu$. It follows that
\begin{equation} \label{eq:B_contam}
\bB_Z = \bB + \eps\sum_{j=1}^g p_j \, [\bdelta_j (\Delta_j - \Deltabar)^T + (\Delta_j - \Deltabar) \, \bdelta_j^T] + \eps^2 \, \sum_{j=1}^g p_j \, (\Delta_j-\Deltabar) \, (\Delta_j-\Deltabar)^T.
\end{equation}
Now suppose that only the first group has been contaminated. This means that $\Delta_j = 0$ for $j \ne 1$, and $\Deltabar = p_1\Delta_1$. Exploiting the fact that $\sum_{j=2}^g p_j \bdelta_j = -p_1\bdelta_1$, we obtain
\begin{align*}
\sum_{j=1}^g p_j \, [\bdelta_j\big(\Delta_j - \Deltabar\big)^T + \big(\Delta_j - \Deltabar\big)\bdelta_j^T] &= p_1\,(1-p_1)\,(\bdelta_1\Delta_1^T + \Delta_1\bdelta_1^T) - p_1^2(\bdelta_1\Delta_1^T + \Delta_1\bdelta_1^T) \\[-2mm]
&= p_1\,(\bdelta_1\Delta_1^T + \Delta_1\bdelta_1^T),
\end{align*}
and $\sum_{j=1}^g p_j\left(\Delta_j-\Deltabar\right)\left(\Delta_j-\Deltabar\right)^T = p_1\,(1-p_1)\,\Delta_1 \, \Delta_1^T$.
Then, up to second-order terms in $\eps$, we can write the contaminated covariance matrices as 
\begin{align}
\bW_Z &= \bW + \Delta\bW = \bW + \eps\,p_1 \, [(\cont\bSigma_1 - \bSigma) + \Delta_1 \, \Delta_1^T] + o(\eps), \nonumber \\[-2.5mm]
\label{eq:cont-oneclass} \\[-2.5mm]
\bB_Z &= \bB\; + \Delta\bB\;= \bB + \eps\,p_1\,(\bdelta_1\Delta_1^T + \Delta_1\bdelta_1^T) + o(\eps). \nonumber
\end{align}
In general, all the contaminated parameters $\expec(\bZ_j)$, $\var(\bZ_j)$, $\expec(\bZ)$, $\bW_Z$, and $\bB_Z$ can be written as their original counterparts ($\expec(\bX_j)$, $\var(\bX_j)$, $\expec(\bX)$, $\bW$, and $\bB$, respectively) plus an additional term each, that quantifies the effect of the contamination.
The additional terms are functions of the percentage of contamination, $\eps$, the probability of the contaminated classes, $p_j$, the parameters of the uncontaminated ($\bmu_j$ and $\bSigma_j$) and contaminating distribution ($\cont\bmu_j$ and $\cont\bSigma_j$).


\subsection{Perturbation analysis of {\rm TR}}\label{subsec:Perturbation}
In Section~\ref{subsec:constscatter}, we have given explicit expressions for the contaminated theoretical between and within covariance matrices, when the contamination happens in one class \eqref{eq:cont-oneclass}. We have stressed that they are of the form $\bB_Z = \bB + \Delta \bB$ and $\bW_Z = \bW + \Delta \bW$, where $\|\Delta\bB\|$, $\|\Delta\bW\| \to 0$ when $\eps \to 0$.
We use this formulation to understand how the TR solution, $\bVtr$, changes when we perturb the pair $(\bB, \bW)$ to $(\bB_Z, \, \bW_Z)$. A generalization of this problem is addressed by Zhang and Yang \cite{zhang2013perturbation}, for arbitrary perturbations such that $\|\Delta\bB\|$, $\|\Delta\bW\|\to 0$. We report their upper bound for the resulting change in $\bVtr$.
By $[\Delta\bB \ \, \Delta\bW]$ we denote the $p \times 2p$ matrix constructed by appending $\Delta\bW$ to $\Delta\bB$.

\begin{theorem}\label{thm:pert} \cite[Thm.~4.1]{zhang2013perturbation}
Let $\bB$ be symmetric and $\bW$ SPD. Let the $p \times k$ matrix $\bVtr$ be the solution to \eqref{eq:TR.MaxProblem} for the pair $(\bB, \bW)$ and suppose that there is a gap $\gamma > 0$ between the $k$th and the $(k+1)$st largest eigenvalues of $\bB-\rho(\bVtr)\,\bW$. Then, when $\eta \to 0$, for all $\|[\Delta\bB \ \, \Delta\bW]\| \le \eta$, there exists $\cont\bVtr$ solution to \eqref{eq:TR.MaxProblem} for the pencil $(\bB+\Delta\bB, \, \bW+\Delta\bW)$ such that 
\[
\|\bVtr \bVtrT-\cont\bVtr\cont\bVtrT\| \le \frac{4\sigma}{\gamma} \, \Big(1+\frac{k}{\tau} \, \|\bW+\Delta\bW\|\Big) + o(\eta),
\]
where $\sigma  = \|\Delta \bB\| + \rho(\bVtr) \, \|\Delta\bW\| \le (1+\rho(\bVtr)) \, \eta$, and $\tau$ is the sum of the $k$ smallest eigenvalues of $\bW$.
\end{theorem}

Note that, as stated in the introduction, we always have $\|\bVtr \bVtrT-\cont\bVtr\cont\bVtrT\| \le 1$. This theorem gives an asymptotic upper bound that is linear in $\eta$, for $\eta \to 0$.

When the contamination level tends to zero, i.e., $\eps\to 0$, we can express $\Delta\bB$ and $\Delta\bW$ as in \eqref{eq:cont-oneclass}. Let us assume that $\cont\bSigma_1 = \bSigma_1$. Then
\[
\|\Delta\bB\| \le 2\, \eps \, p_1 \, \|\bdelta_1\| \, \|\Delta_1\| + o(\eps), \qquad \|\Delta\bW\| \le \, \eps \, p_1^2 \, \|\Delta_1\|^2 + o(\eps).
\]
Thus,
\[
\|[\Delta\bB \ \, \Delta\bW]\| \le \,\eps \, p_1\, \|\Delta_1\|\, (2\,\|\bdelta_1\|\, + \, p_1 \, \|\Delta_1\| ) + o(\eps).
\]
If $p_1$, $\|\bdelta_1\|$, and $\|\Delta_1\|$ are fixed, then we have just shown that $\|[\Delta\bB \ \, \Delta\bW]\| \le \eps\,C$, where $C= p_1\, \|\Delta_1\|\, (2\,\|\bdelta_1\|\, + \, p_1 \, \|\Delta_1\| )>0$. If we set $\eta= \eps\,C$ in Theorem~\ref{thm:pert}, then $o(\eta) = o(\eps)$. Thus we can write the upper bound for the dissimilarity measure between $\bVtr$ and $\cont\bVtr$:
\begin{align*}
\|\bVtr \bVtrT-\cont\bVtr\cont\bVtrT\| & \le \frac{4}{\gamma} \, \eps \, p_1 \, \|\Delta_1\| \, \big(2\,\|\bdelta_1\| + p_1 \, \rho(\bVtr) \, \|\Delta_1\|\big)\\
&\qquad\qquad\cdot \Big(1+\frac{k}{\tau} \, \left(\|\bW\| + 2\, \eps \, p_1^2 \, \|\Delta_1\|^2\right) \Big) + o(\eps)\\
& = \frac{4}{\gamma} \, \eps \, p_1 \, \|\Delta_1\| \, \big(2\,\|\delta_1\| + p_1 \, \rho(\bVtr) \, \|\Delta_1\|\big)\, \Big(1+\frac{k}{\tau} \, \|\bW\|\Big) + o(\eps),
\end{align*}
where in the last step we discard a new $o(\eps)$-term.
Let $\lambda_{p-k+1}(\bW) \ge \cdots \ge \lambda_p(\bW)$ be the $k$ smallest eigenvalues of $\bW$; then $\lambda_1(\bW)=\|\bW\|$ and $\tau=\sum_{i=1}^{k}\lambda_{p-k+i}(\bW)$. Since $\tau \ge k\,\lambda_p(\bW)$, we can write:
\begin{equation}\label{eq:Upperlim_TR}
\|\bVtr \bVtrT-\cont\bVtr\cont\bVtrT\|
 \le \frac{4}{\gamma} \, \eps \, p_1 \, \|\Delta_1\|\left(2\,\|\delta_1\| + p_1 \, \rho(\bVtr) \, \|\Delta_1\|\right)\, \Big(1+ \frac{\lambda_1(\bW)}{\lambda_p(\bW)}\Big) + o(\eps).
\end{equation}
The quotient $\lambda_1(\bW)/\lambda_p(\bW)$ is the condition number of $\bW$, which is an indication of the ill-conditioning of $\bW$. In the situation of a high condition number, even small perturbations in the $\bW$ matrix can lead to big changes in the subspaces generated by TR.

The upper limit \eqref{eq:Upperlim_TR} is linear in the distance between $\bmu_1$ and $\bmu$, i.e., $\|\bdelta_1\|$, but quadratic in $\|\Delta_1\|$, the distance between the non-contaminated mean vector and the contaminated mean vector. Thus $\|\Delta_1\|$, which is related to the contamination process, is the term that contributes the most to increasing the upper limit of this dissimilarity index. Moreover, the upper limit depends explicitly on the theoretical within covariance matrix, $\bW$, but only implicitly on $\bB$ (via $\rho(\bVtr)$ and $\gamma$), assigning different levels of importance to these two matrices.


A way to control the impact of contamination in the estimated TR and FDA coefficients is to consider robust estimators, which are discussed in the next section.

\section{Robust trace ratio and robust Fisher discriminant analysis }\label{sec:rob}

In this section, we propose a robust version of the TR algorithm, and also study a robust version of FDA. 
Instead of considering the classical estimators of the between and within covariance matrices defined in \eqref{eq:SbSw}, we use their robust versions based on the minimum covariance determinant (MCD) estimators of location and covariance matrix, described in \cite{rousseeuw1999fast}. Let us denote the robust estimators by $\wt{\bmu}_j$ and $\wt{\bS}_j$ for the location and the covariance matrix, respectively, of the $j$th group, $j = 1, \dots, g$. The global location estimator is defined as $\wt{\bmu} = \sum_{j=1}^g \tfrac{n_j}{n} \, \wt{\bmu}_j$. The robust versions of the between and within scatter matrices are then defined as
\begin{equation}\label{eq:sbsw-rob}
\wt{\bB} = \sum_{j=1}^g \tfrac{n_j}{n} \, (\wt{\bmu}-\wt{\bmu}_j) \, (\wt{\bmu}-\wt{\bmu}_j)^T, \qquad \wt{\bW} = \sum_{j=1}^g \tfrac{n_j-1}{n-g} \, \wt{\bS}_j.
\end{equation}
We now propose a robust version of the trace ratio method, which we denote by rTR. 
For this method, we determine $\wtbVtr$ from Algorithm~1 in Subsection~\ref{Sec:TR optimization}, but now using as input the robust estimates of the between and within scatter matrices, $\wt{\bB}$ and $\wt{\bW}$, respectively.

We compare rTR with a robust version of FDA (rFDA), where a robust coefficient matrix, $\wtbVfda$, is obtained by finding the first $k$ generalized eigenvectors of the pencil $(\wt{\bB}, \wt{\bW})$. 
For the classical estimates, the statistical literature usually recommends the derivation of the eigenvectors of the matrix ${\bW}^{-1}{\bB}$, as can be confirmed in, e.g., \cite{johnson2007applied} or \cite{Hastie2009StatLearning}.
Nevertheless, we note that standard methods working on this pair (such as the ones implemented in the R package {\sf geigen}, see \cite{R.geigen:2019} for details) are numerically preferable over approaches based on the transformed matrix $\bW^{-1} \bB$ (see, e.g., \cite[Ch.~8]{golubvanloan}). Therefore, we use the function {\sf geigen} from \cite{R.geigen:2019} to derive the classical and robust FDA matrix of coefficients. 

The procedure to classify a test set observation, $\bx_0$, into one of the groups, for both rTR and rFDA, is as follows.
\begin{enumerate}
\item[(i)] Project the training set observations on the subspace spanned by $\wt{\bV}$ (estimated by rTR or rFDA).
\item [(ii)] Compute the MCD estimates of the group projected locations $\wt {\bmu}_{Uj}$
and the covariance matrices, and from these the associated $k \times k$ pooled covariance matrix $\wt\bW_{U}$.
\item[(iii)] Compute $\bu_0=\wt{\bV}^T\bx_0$ and the Mahalanobis distance between $\bu_0$ and $\wt {\bmu}_{Uj}$, $\|\bu_0-\wt{\bmu}_{Uj}\|_{\wt\bW_{U}^{-1}}$.
\item[(iv)] Assign $\bx_0$ to the group $j$ that minimizes the quantity $\|\bu_0-\wt{\bmu}_{Uj}\|_{\wt\bW_{U}^{-1}}^2 - 2\, \log (n_j/n)$.
The second term balances the class prior information in the decision rule.
\end{enumerate}
Step (iv) is equivalent to performing robust LDA over the projected data on the reduced subspace, spanned by a robust estimate of $\bV$. The robust LDA is obtained by plugging the MCD estimators of location and covariance in the Mahalanobis distances, computed in the reduced subspace.
This is a robust way to mitigate the potentially harmful effect of outliers in the classification rules based on the Euclidean distance in the projected space. We remark that, in case of high-dimensional data, we may also use regularized MCD \cite{boudt2020mrcd} or any other robust estimators of the groups location and covariance matrix.

In the next sections, we compare FDA and TR when either classical or robust estimates for the covariance matrices are considered.

\section{Simulation study} \label{sec:simulation}
We first explore TR and FDA as classifiers. Here the dimensionality reduction is embedded in the classifier: first, we build the projection matrix, either $\bVfda$ or $\bVtr$, and then we classify the projected data points based on the distances from the projected group means. 
The simulation scheme presented in this section has two main objectives: (i) compare TR with FDA classical and robust classifiers, under clean and contaminated data, and (ii) analyze the performance of TR and FDA classifiers under the presence of irrelevant variables for the associated classification problem. 

\subsection{Scenarios}
\label{sec:scenarios}
In all scenarios, we consider four groups ($g=4$), each one characterized by a random vector of dimension $p=3$, except for the last scenario, where we continue to have three relevant variables for the classification problem, but we add $q \in \{0,10,20,50,100\}$ other irrelevant variables. 
The contamination scheme is the one described in Section~\ref{subsec:constscatter}. 
More precisely, in the $\ell$th scenario, the clean $j$th group is $\bX_j = (\bX\mid Y = j) \sim \caln_p\big(\bmu_j^{(\ell)} \, , \, \bSigma_j^{(\ell)}\big)$ and the contaminating variable is distributed as $\cont\bX_j \sim \caln_p\big({}^{c}\bmu_j^{(\ell)} \, , \, \bSigma_j^{(\ell)}\big)$, for $j=1, \dots,g$, $\ell = \scen{I}, \dots, \scen{IV}$. 
Contamination is absent when $\eps=0$.

For all scenarios, we choose the group mean vectors to suggest that two discriminant variables are enough to separate the groups, thus we choose $k=2$ as the reduced dimension. To make the scenarios more realistic, the covariance matrices challenge this separation.

As a benchmark for each scenario, we use the associated parameters of the regular population ($\bmu_j^{(\ell)}$ and $\bSigma_j^{(\ell)}$) to write the theoretical between and within scatter matrices. Using these as input, we find the solutions to \eqref{eq:gep} and \eqref{eq:TR.MaxProblem}, and refer to them as theoretical matrices of coefficients, which are indicated by ${\bf V}_{\rm FDA}$ and ${\bf V}_{\rm TR}$, respectively. 

\begin{table}
\small
\caption{Parameters characterizing regular and contaminated distributions for each simulation scenario ($\ell = \scen{I}, \dots, \scen{IV}$), where $\bM^{(\ell)}= [\bmu_1^{(\ell)}, \dots, \bmu_4^{(\ell)}]$, $^c\bM^{(\ell)}= [{}^c\bmu_1^{(\ell)}, \dots, {}^c\bmu_4^{(\ell)}]$, and $\bSigma_j^{(\ell)}$ is the covariance matrix of the $j$th group in the $\ell$th scenario. }
\label{tab:scheme}
\begin{tabular}{ccccccc}
\toprule
$\ell$ & $\bM^{(\ell)}$ & ${}^c\bM^{(\ell)}$ & $\bSigma_1^{(\ell)}$ & $\bSigma_2^{(\ell)}$ & $\bSigma_3^{(\ell)} $& $\bSigma_4^{(\ell)}$ \\ 
\midrule\\
\addlinespace[-2ex]
\scen{I} & $\smtxa{rrrr}{
 15 & 15 & & \\
 3 & -3 & & \\
 & & 2 & -2} $ &
$\smtxa{rrrr}{
 15 & 15 & & \\
 -27 & -3 & & \\
 & & 2 & -2} $ &
$\smtxa{lll}{1 & & \\
 & \ph-1 & -0.25 \\
 & -0.25 & \ph-1} $ &
$\bSigma_{1}^{(\scen{I})}$ &
$\bSigma_{1}^{(\scen{I})}$ &
$\bSigma_{1}^{(\scen{I})}$ \\
\addlinespace[0.5ex]
\midrule\\
\addlinespace[-2ex]
\scen{II} & $\smtxa{rrrr}{
 & & 3 & -3 \\
 -3 & 3 & & \\
 & & 1 & 1}$ &
$\smtxa{rrrr}{
 & & 3 & -3 \\
 -3 & -27 & & \\
 & & 1 & 1}$ &
$\smtxa{rrrr}{
 1 & & \\
 & 3 & \\
 & & 1 } $&
$\bSigma_{1}^{(\scen{II})}$ &
$\smtxa{rrrr}{
 3 & & \\
 & 1 & \\
 & & 3} $&
$\bSigma_{3}^{(\scen{II})}$ \\
\addlinespace[0.5ex]
\midrule\\
\addlinespace[-2ex]
\scen{III} & $\smtxa{rrrr}{ & 10 & & -10 \\
-3 & & 3 & \\
 1& &1 & }$ &
$\smtxa{rrrr}{ & 10 & & -10 \\
-3 & & -27 & \\
 1& &1 & }$ &
$\bSigma_{1}^{(\scen{II})}$ &
$\bI_3$ &
$\smtxa{lll}{1 & & \\
 & \ph-3 & -0.5 \\
 & -0.5 & \ph-1} $&
$\bI_3$ \\
\addlinespace[0.5ex]
\midrule\\
\addlinespace[-2ex]
\scen{IV} & $ \smtxa{c}{
 \bM^{(\scen{II})} \\
 \zero_{q\times 4}} $ & 
$ \smtxa{c}{
 {}^c\bM^{(\scen{II})} \\
 \zero_{q\times 4}} $&
$\smtxa{rrrr}{\bSigma_{1}^{(\scen{II})} & \zero_{3\times q} \\
 \zero_{q\times 3} & \bI_q} $& 
$\bSigma_{1}^{(\scen{IV})}$ &
$\smtxa{rrrr}{\bSigma_{3}^{(\scen{II})} & \zero_{3\times q} \\
\zero_{q\times 3} & \bI_q} $&
$\bSigma_{3}^{(\scen{IV})}$ \\
\addlinespace[1ex]
\bottomrule
\end{tabular}
{\footnotesize \textbf{Note:} $\zero_{q\times q'}$ represents a $q \times q'$ matrix of zeros. In scenario \scen{IV}, $q \in \{0,\,10,\,20,\,50,\,100\}$ represents the number of irrelevant variables. Omitted matrix entries represent zeros.}
\end{table}

\begin{figure}[!ht]
\begin{subfigure} {0.5 \textwidth} %
\includegraphics[width = 0.93\textwidth]{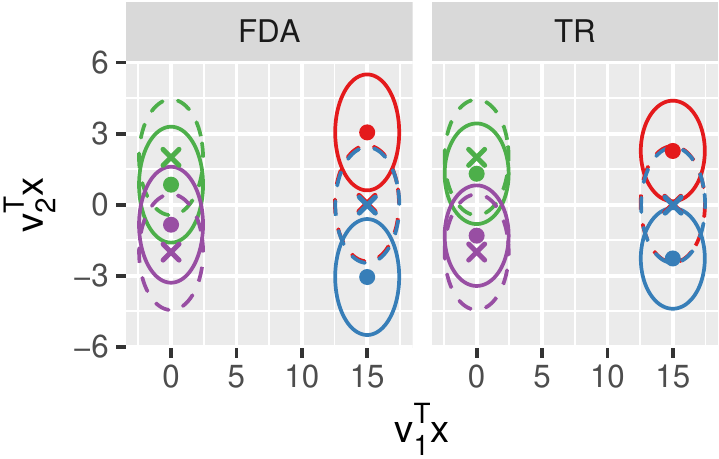} 
 \caption{Scenario \scen{I}.}
 \label{fig:Ellipse_I}
\end{subfigure} 
\begin{subfigure}{0.5 \textwidth}%
 \includegraphics[width = 0.93\textwidth]{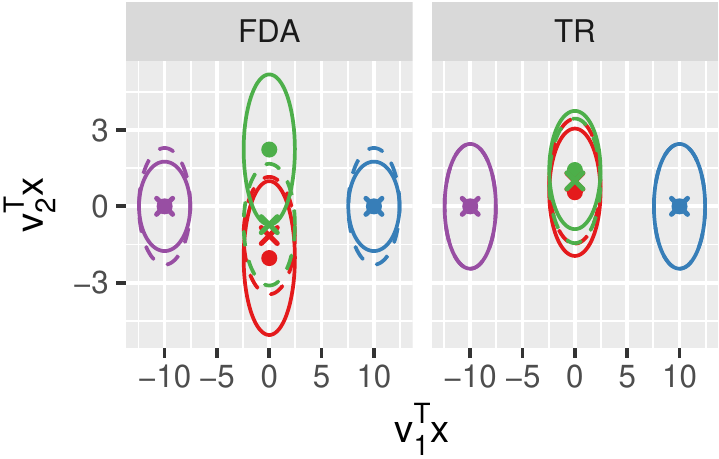}%
\caption{Scenario \scen{III}.}
 \label{fig:Ellipse_III}
\end{subfigure}
\hfill
\begin{subfigure}{\textwidth}%
 \centering
 \includegraphics[width = \textwidth]{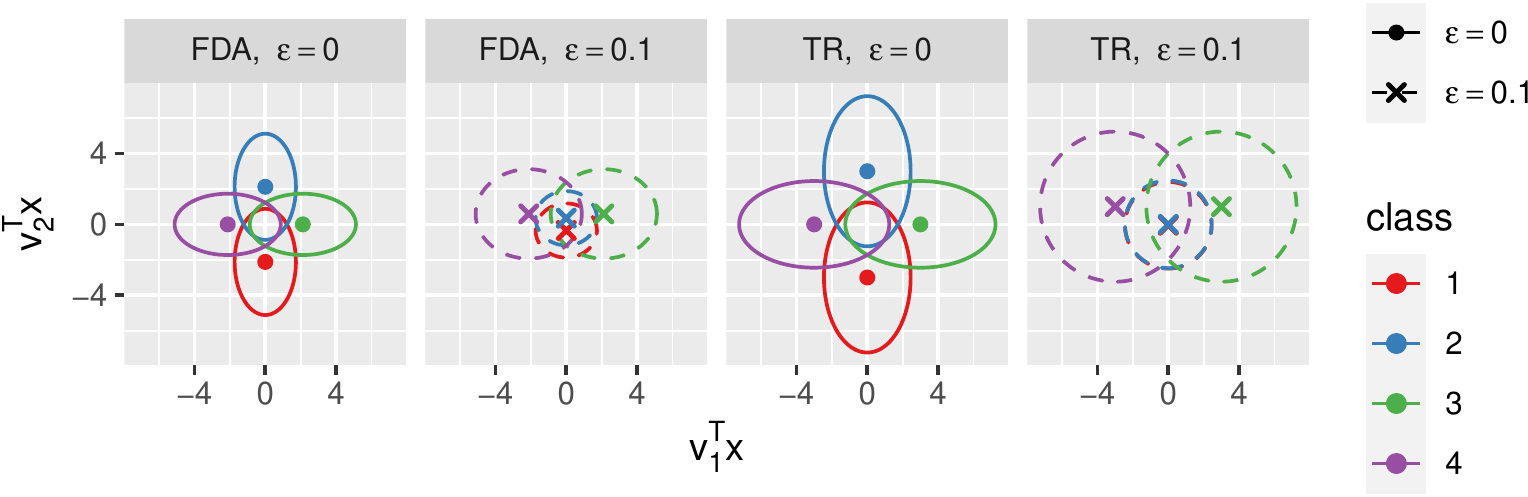}%
 \caption{Scenario \scen{II}.}
 \label{fig:Ellipse_II}
\end{subfigure}
\caption{Mean vectors and 95\% confidence ellipses for the projected data assuming that the projected data follows a bivariate normal distribution. The projection matrices are the theoretical $\bVtr$ and $\bVfda$, corresponding to the contamination levels $\eps = 0$ (solid line) and $\eps = 0.1$ (dashed line).}
 \label{fig:Ellipse_all_scenarios}
\end{figure}

When all group covariance matrices are equal, i.e., in the first scenario ($\ell = \scen{I}$), and the underlying group population follows a multivariate normal distribution, the classification rule that minimizes (maximizes) the total probability of misclassification (probability of correctly assign an observation to a group, called theoretical accuracy) is linear (LDA, cf.~Section~\ref{sec:fda-lda}) and would coincide with FDA, if we considered $k=3$. In the present work, we only consider $k=2$, which does not enjoy of the same optimality property. 
In addition, we see that TR leads to better results than FDA in this scenario, showing that the missing dimension has a relevant impact on the total probability of misclassification. 
The solutions to \eqref{eq:gep} and \eqref{eq:TR.MaxProblem} associated with the parameters characterizing the regular distribution are:
\begin{equation}\label{eq:V_scenario O}
{\bV}^{(\scen{I})}_{\rm FDA} = \smtxa{cl}{
1 & 0 \\
0 & 1 \\
0 & 0.415},	\qquad			
{\bV}^{(\scen{I})}_{\rm TR} = \smtxa{cl}{
1 & 0 \\
0 & 0.757 \\
0 & 0.654}.	
\end{equation}
In this case, the first two FDA discriminant variables are $X_1$ and a linear combination of $X_2$ and $X_3$, where $X_2$ has a higher weight than $X_3$.
The TR matrix of coefficients has a similar structure, but gives a more balanced weight to $X_2$ and $X_3$.
It should be noted that these weights cannot be compared one-to-one because the classification is based on a different metric: while the Mahalanobis distance is used for FDA, the Euclidean norm is exploited for classification based on the trace ratio method.
    %
    %
    
To get a better understanding, we visualize the effect of the contamination on the projected mean vector and covariance matrix in each scenario. In Figure~\ref{fig:Ellipse_all_scenarios} we draw 95\% confidence ellipses for the projected data, assuming that the projected data follows a bivariate normal distribution with mean vector and covariance matrix equal to the theoretical values under each scenario. We consider $\eps=0$ and $\eps=0.10$; in the latter case, the projected contaminated group does not follow a bivariate normal distribution, 
so the 95\% confidence ellipse would not be an appropriate confidence region. Nevertheless, this approximation serves its purpose. 

For scenario $\ell = \scen{I}$, the contamination strongly affects the first and second group mean vectors, leading to an overlap between these two groups in the projected space spanned by $\bVfda^{\scen{(I)}}$ and $\bVtr^{\scen{(I)}}$. As a side effect, the separation between the third and fourth group increases for both FDA and TR, as shown in Figure~\ref{fig:Ellipse_I}. 

In the following scenarios, we consider different class covariance matrices. 
The second scenario ($\ell = \scen{II}$) describes a case where TR and FDA show a similar performance. The covariance matrices are all diagonal (see Table~\ref{tab:scheme} for more details). The theoretical matrices of coefficients are \begin{equation}\label{eq:V_scenario I}
{\bV}^{(\scen{II})}_{\rm FDA} ={\bV}^{(\scen{II})}_{\rm TR} = \smtxa{cc}{
1 & 0 \\
0 & 1 \\
0 & 0 \\
}.	
\end{equation}
As illustrated in Figure~\ref{fig:Ellipse_II}, in the projected space we cannot distinguish between groups 1 and 2, but the contamination also has a considerable impact on the other two groups, mainly on the shapes of the covariance matrices, as illustrated by the changes in the 95\% confidence ellipses. 

In the third scenario ($\ell = \scen{III}$), we illustrate an example where FDA leads to better results than TR. Like in the previous scenario, by looking at the group means (cf.~Table~\ref{tab:scheme}), $X_1$ and $X_2$ seem to be enough to separate the groups. 
However, the correlation between $X_2$ and $X_3$ in the third group affects $\bv_2$, where both variables have a nonzero weight. We can anticipate that FDA performs better than TR since it gives higher weight to $X_2$ than to $X_3$, as follows: 
\begin{equation}\label{eq:V_scenario II}
{\bV}^{(\scen{III})}_{\rm FDA} = \smtxa{cl}{
1 & 0 \\
0 & 1 \\
0 & 0.125},	\qquad			
{\bV}^{(\scen{III})}_{\rm TR} = \smtxa{cl}{
1 & 0 \\
0 & 0.147\\
0 & 0.989}.	
\end{equation}
This is reinforced by Figure~\ref{fig:Ellipse_III}, where the 95\% confidence region of the first and third group show a high overlap in the subspace spanned by TR, when no contamination is present. Contamination increases this overlap but does not seem to affect groups 2 and 4. For FDA, we observe a considerable separation between groups 1 and 3 in the projected space, when no contamination occurs. 
The contamination scheme substantially increases the overlap between groups 1 and 3 in the projected space, as can be confirmed by Figure~\ref{fig:Ellipse_III}. Additionally, contamination affects the variability of groups 2 and 4. 


In the fourth scenario ($\ell = \scen{IV}$) we test the performances of FDA and TR in presence of irrelevant variables in the data. 
In \cite{ortner2020robust} a simulation scheme is designed to study the effect of increasing the number of irrelevant variables for the classification problem, when keeping the number of relevant variables and the group sample sizes fixed. The authors have shown that, for clean and contaminated data, the accuracy of both classical and robust linear discriminant classifiers decreases considerably, as the number of irrelevant variables increases. To verify whether this behavior also occurs in TR and FDA, we design the fourth scenario ($\ell = \scen{IV}$) by adding $q\in\{0,10,20,50,100\}$ irrelevant variables to the second scenario. We choose scenario \scen{II} as starting point, since it is the one with no significant differences in the median accuracy between theoretical TR and FDA (cf.~Section~\ref{sec:simulexp}).
The irrelevant variables are independent and identically distributed with zero mean and unitary variance; they are also uncorrelated from the potentially relevant variables, i.e., $X_1$, $X_2$, and $X_3$. The parameters associated with the fourth scenario are listed in Table~\ref{tab:scheme}. From the theoretical viewpoint, the relevant variables are well captured by the matrices of coefficients: 
\begin{equation}\label{eq:V_scenario IV}
{\bV}^{(\scen{IV})}_{\rm FDA} ={\bV}^{(\scen{IV})}_{\rm TR} = \smtxa{cc}{
1 & 0 \\
0 & 1 \\
0 & 0 \\[-1mm]
\vdots & \vdots\\
0 & 0 \\
}.	
\end{equation}	
In the following section, we are going to generate data from the scenarios in Table~\ref{tab:scheme} to study the sample behavior of $\bVtr$ and $\bVfda$ under different contamination levels and different estimators for the between and within covariance matrices.

\subsection{Experimental setting}
\label{sec:simulexp}
Given a scenario, for each experiment we generate a training set of $n_j =400$ data points for each class, and a test set composed by $n_j=40$ observations per group, according to the parameters defined in Table~\ref{tab:scheme}. For scenario \scen{IV} we also consider $n_j = 40$ for the training set. Training sets and test sets are independently generated. We study the different contamination levels $\eps \in \{0,\, 0.05,\, 0.10,\, 0.15,\, 0.20,\, 0.25,\, 0.30 \}$. 
We test the performances of the following classifiers.
\begin{itemize}
\item cTR and cFDA: TR and FDA with classical estimators for the covariance matrices. The associated classification rules are described in Section~\ref{sec:ldatr}.
\item rTR and rFDA: TR and FDA with robust estimators for the covariance matrices, as described in Section~\ref{sec:rob}. These are obtained by the MRCD robust location and covariance matrix estimators via the {\sf CovMrcd} routine from the R package {\sf rrcov}, where the parameter ${\sf alpha}=0.75$ controls the size of the subsets over which the covariance determinant is minimized (see \citep{Hubert.VanDriessen:2004,Todorov.rrcov:2009} for further details). The classification rules are described in Section~\ref{sec:rob}. Given that the MRCD estimates select $75\%$ of the data, we expect that they are able to deal with datasets with at most $25\%$ of outliers. Finally, we remark that the regularization parameter of the MRCD estimates is chosen adaptively, so that MRCD estimates coincide with MCD estimates when regularization is not needed \cite{boudt2020mrcd}.
\item tTR and tFDA: theoretical TR and FDA, where the covariance matrices are computed analytically (cf.~\eqref{eq:Var(X)}). In this case, no training set is needed. The associated classification rules are described in Section~\ref{sec:ldatr} and make use of the theoretical group means.
\item tQDA: quadratic discriminant classification rule (see, e.g., \cite{Hastie2009StatLearning}) for multivariate normal data. The parameters of the distributions are known, so we do not need a training set.
\item AdaBoost: nonlinear classifier (with decision trees as the weak learners) \citep{Freund.Schapire:1997}. This is abbreviated by {\sf ab} in the figures. We have selected AdaBoost due to its popularity (and proven performance) in the Machine Learning community.
\end{itemize}
We remark that tQDA and AdaBoost do not involve any linear dimensionality reduction strategy. When the groups have a multivariate normal distribution with different covariance matrices, tQDA is the optimal classifier, because it maximizes the accuracy (or, equivalently, it minimizes the total probability of misclassification). In scenario $\scen{I}$, the covariance matrices are all equal, so QDA coincides with LDA, which is the optimal linear classifier.

In cTR, cFDA, rTR, and rFDA, the training set is used to find a projection matrix for each method, while we assess the performances of the classification rules on the test set. Each experiment is repeated $200$ times, meaning that for each contamination level $\eps$, we generate $200$ training sets and $200$ test sets. To get trustful performance measures, the test set is never contaminated.

For each algorithm and scenario, we report the sample accuracy.
We also evaluate the goodness of a sample subspace $\wh\bV$ compared to its theoretical counterpart $\bV$. To do so, we compute the angles between these two subspaces $\angle (\bV, \wh\bV)$. The matrices $\bV$ and $\wh\bV$ need to have orthogonal columns: these can be obtained from the solution to FDA via a QR decomposition. The solution to TR already has orthogonal columns. Angles compare only subspaces: they indicate the quality of the estimates of the theoretical coefficients obtained with FDA and TR.

\subsection{Results and discussion}\label{SubSec:Discussion}
We now report the results of the simulation study. Table~\ref{tab:summary} summarizes the ranking of the algorithms based on the median accuracy over the generated test datasets. Relative positions in Table~\ref{tab:summary} are tested via the Wilcoxon signed-rank test for paired samples. Except for AdaBoost, the comparison is made only under the hypothesis that the
mean vectors and covariance matrices
are known and not contaminated. As we will see, the simulation study agrees with the considerations that we have made in Section~\ref{sec:scenarios}, when no contamination occurs: in the first scenario, tTR shows a better accuracy than tFDA; in the second and fourth scenario, the two methods are equivalent; in the third scenario, tFDA gives better results than tTR.
\begin{table}[ht!]
\centering\small
\caption{Ranking of theoretical methods and AdaBoost, with $\eps = 0$, according to the conclusions of the Wilcoxon signed-rank tests for paired samples, with the False Discovery Rate correction \citep{Benavoli.et.al:2016}. The median accuracy is reported in brackets.}
\begin{tabular}{lrrrrrrr}
\toprule
\scen{I}& tQDA {\footnotesize (0.994)}& $>$ & AdaBoost \footnotesize{(0.988)} & $>$ & tTR \footnotesize{(0.969)} &$>$ & tFDA \footnotesize{(0.900)}\\
\scen{II, \, IV}& tQDA \footnotesize{(0.894)}& $>$ & tFDA \footnotesize{(0.881)} & $=$ & tTR \footnotesize{(0.881)} &$>$ & AdaBoost \footnotesize{(0.875)}\\
\scen{III}& tQDA \footnotesize{(0.981)}& $=$ & tFDA \footnotesize{(0.981)} & $>$ & AdaBoost \footnotesize{(0.975)} &$>$ & tTR \footnotesize{(0.838)} \\
\bottomrule
\end{tabular}
\label{tab:summary}
\end{table}

Figure~\ref{fig:angles-samples} shows the angles between the benchmark matrices $\bVtr$ or $\bVfda$ and the solution to TR or FDA when $\bB$ and $\bW$ are estimated from a sample. We consider both classical and robust estimates. Since we are able to compute the theoretical contaminated matrices $\cont \bB$ and $\cont \bW$ from Section~\ref{sec:Contam effect on W and B}, we can show the theoretical angles between the original subspaces and the contaminated ones (represented in gray in Figure~\ref{fig:angles-samples}). 
\begin{figure}[ht]
\centering
\includegraphics[width = \textwidth]{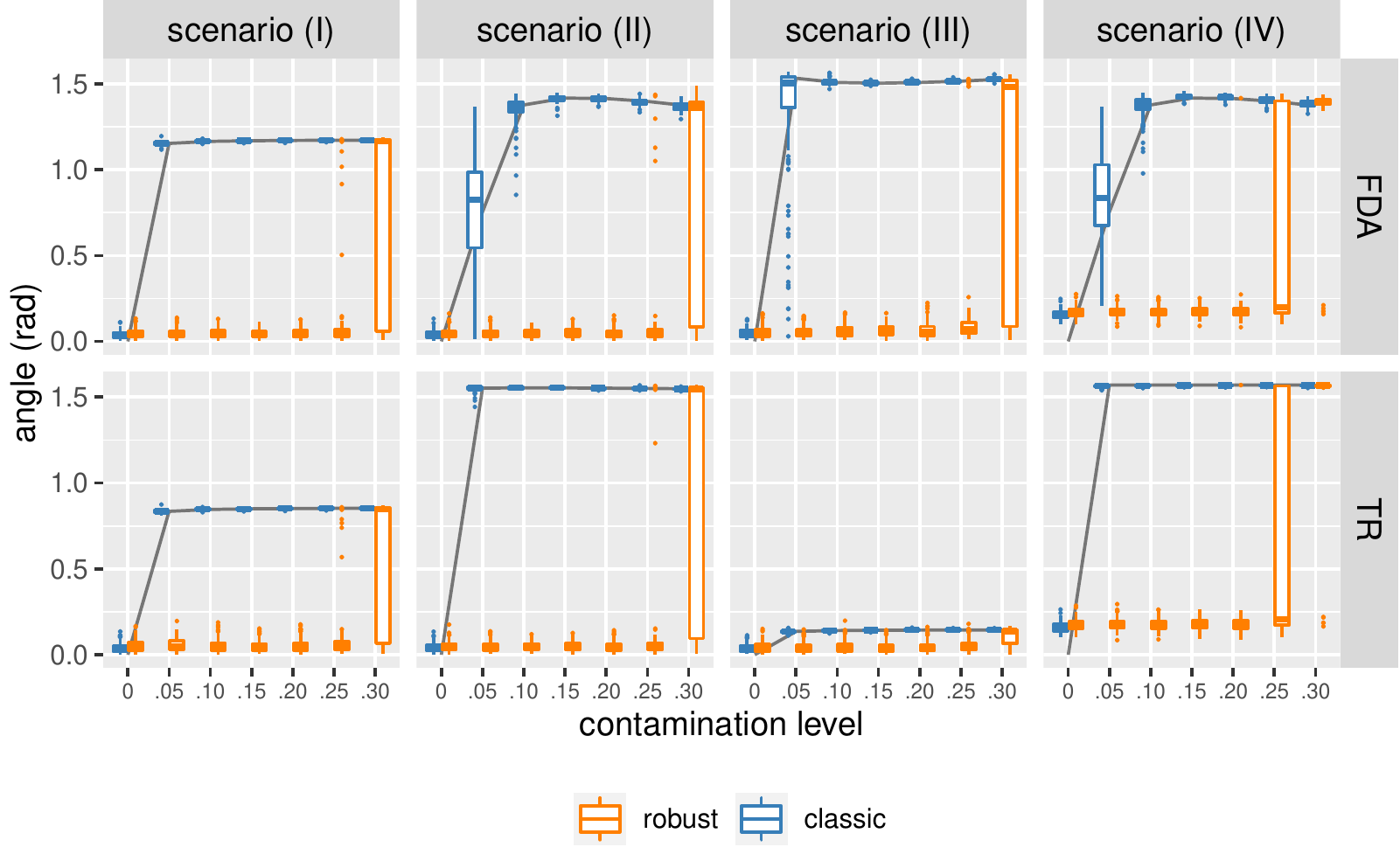}
\caption{Angle between the solution to TR (FDA) when the $\bB$ and $\bW$ are the theoretical ones, and the solution to TR (FDA) in presence of contamination when $\bB$ and $\bW$ are estimated from a sample. We consider both classical (blue) and robust (orange) estimates. The gray lines correspond to the theoretical angle between the original subspace with and without contamination. The horizontal axis shows different levels of contamination in the sample. For scenario $\scen{IV}$, we only consider $q = 10$.}
\label{fig:angles-samples}
\end{figure}

Figures \ref{fig:Boxplots_I} to \ref{fig:Boxplots_III} report the boxplots of the accuracy for each method and level of contamination. We indicate as benchmark results those related to tTR, tLDA, and tQDA, since these are not affected by the contamination level. Given the theoretical findings of the contaminated scatter matrices, derived in Section~\ref{subsec:constscatter}, it is possible to obtain the accuracy for the theoretical tFDA and tTR for each level of contamination. The results based on the theoretical derivations and solely on simulations are quite similar, so only cFDA and cTR (simulation-based) are shown.

In the first scenario ($\ell = \scen{I}$), and when no contamination is considered, the Friedman test \citep{Demsar:2006} concludes that there are statistically significant differences between tQDA, AdaBoost, tTR, and tFDA. 
Moreover, Table \ref{tab:summary} reports the following ranking of the methods, based on Wilcoxon signed-rank test for the median accuracy: tQDA, AdaBoost, tTR, and tFDA.
This ranking is also supported by the plots in Figure~\ref{fig:Boxplots_I}.
In this case, it is known that tQDA reduces to LDA, which would coincide with FDA if we considered $k=3$ instead of $k=2$. FDA with $k=2$ loses almost 10\% in median accuracy by not considering a third dimension on the projected space, while tTR leads to higher median accuracy than tFDA, when $k=2$. In addition, AdaBoost turns out to be the second best method.
Under contamination, we observe that the robust methods can accommodate until 25\% of contamination while the classical counterparts break down, presenting an increase in variability. 
From Figure~\ref{fig:angles-samples}, we see that the subspaces obtained with cFDA and cTR change as soon as $\eps \ge 0.05$; cFDA changes slightly more than cTR. The subspaces generated by rFDA and rTR stay close to their theoretical counterparts, as expected, although the robust methods also start breaking down at $\eps = 0.25$.
\begin{figure}[htb!]
\centering
\includegraphics[width=\linewidth]{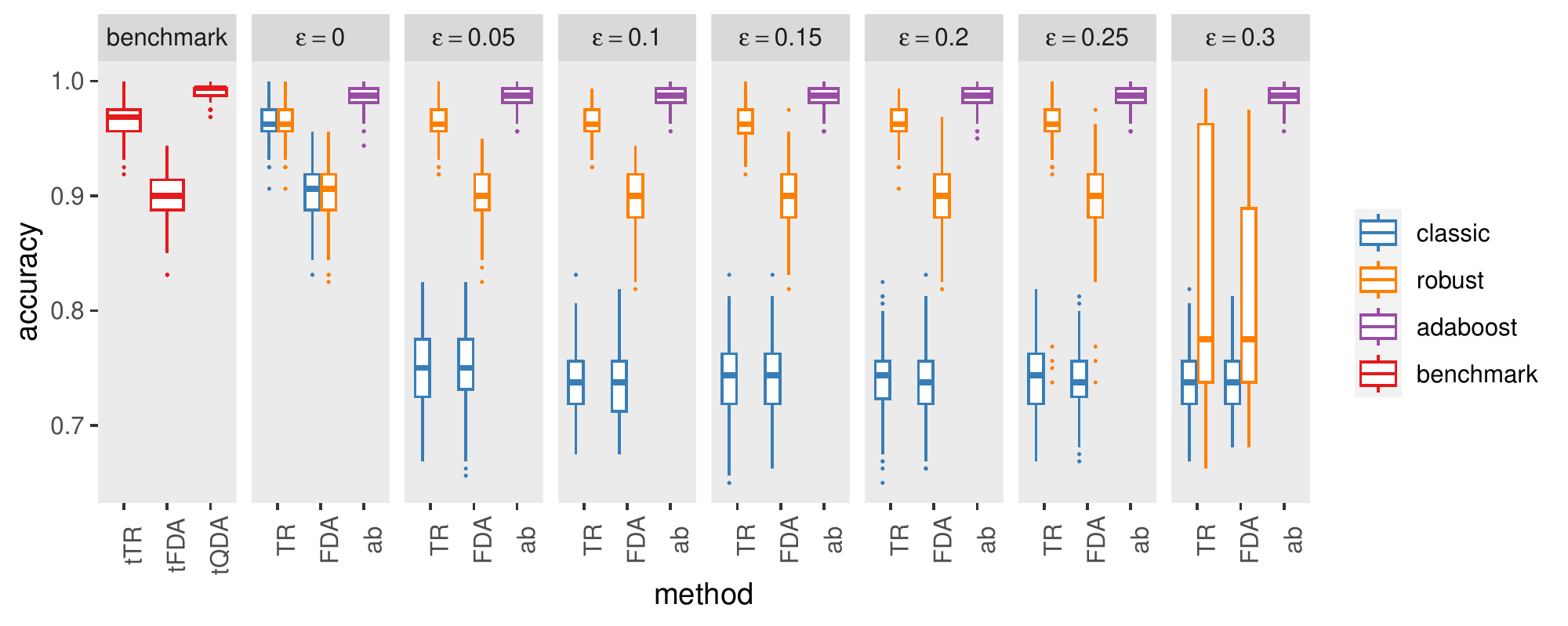}
\caption{
Distribution of the accuracy over 200 simulation runs of Scenario \scen{I}. The accuracy of tTR, tFDA, and tQDA is reported as benchmark in the leftmost plot. Classical/robust TR and FDA classifiers, and AdaBoost, are considered under several levels of contamination.}
\label{fig:Boxplots_I}
\end{figure}

In the second scenario ($\ell = \scen{II}$), when $\eps = 0$, the Friedman tests followed by the Wilcoxon signed-rank test (in Table \ref{tab:summary}) lead to the following order of the methods: tQDA, tFDA and tTR with (statistically) the same median accuracy, AdaBoost;
this is agrees with Figure~\ref{fig:Boxplots_II}. 
We stress that this is an example where linear approximations (tTR and tFDA) of the optimal quadratic rule lead to better results than a nonlinear classification algorithm (AdaBoost). When the data is not contaminated, classical and robust linear estimators perform similarly to their theoretical counterparts, and better than AdaBoost. When we increase the proportion of contamination, Figure~\ref{fig:Boxplots_II} shows that cTR is immediately affected by the contamination, while the accuracy of cFDA gradually decreases, but ends in the lowest median accuracy for $\eps = 0.30$, compared to all other methods. Both robust alternatives are unaffected by small to moderate percentages of contamination, but break down for $\eps=0.30$, as expected. AdaBoost does not seem to be affected by the contamination. 
In Figure~\ref{fig:angles-samples}, the subspaces obtained with cFDA and cTR change as soon as $\eps \ge 0.05$; cTR changes more than cFDA, which also shows more variability than cTR for $\eps = 0.05,\,0.1$. As expected, the subspaces generated by rFDA and rTR are close to their theoretical counterparts; we observe their breakdown starting from $\eps = 0.25$.
\begin{figure}[htb!]
\centering
\includegraphics[width=\linewidth]{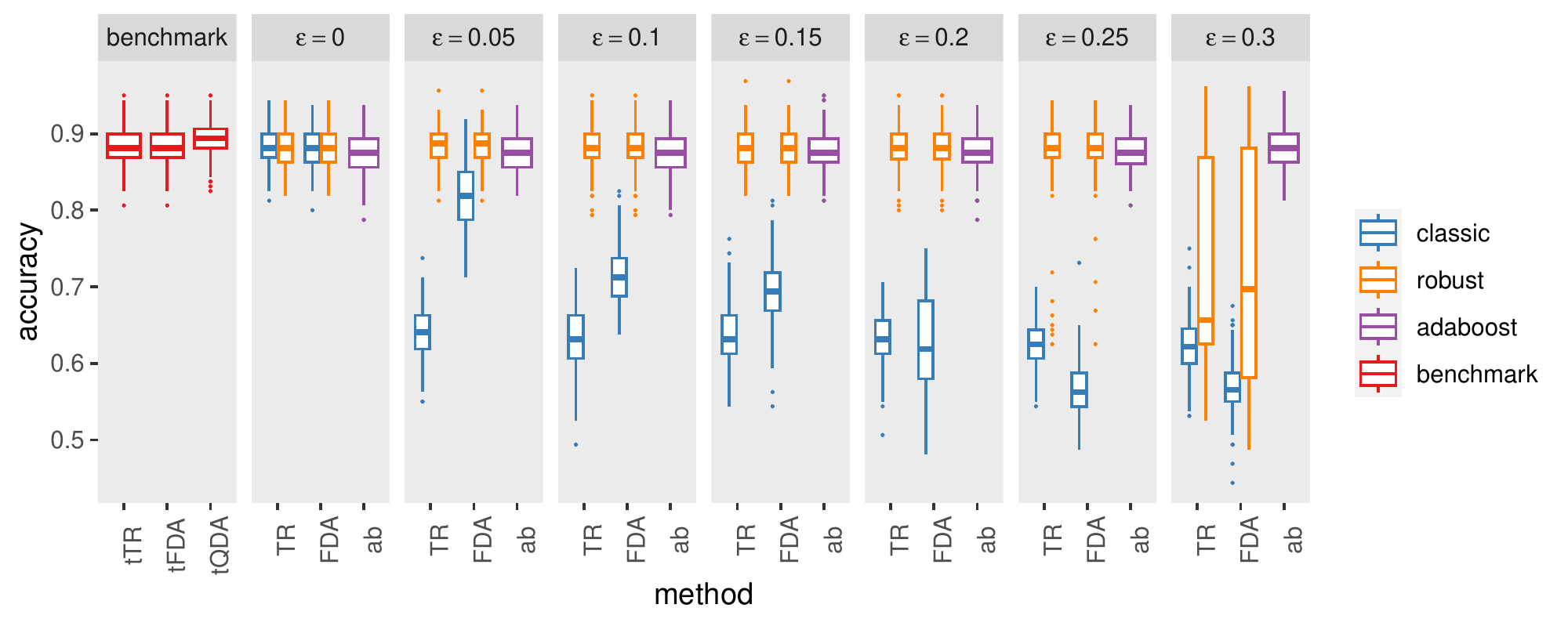}
\caption{Distribution of the accuracy over 200 simulation runs of Scenario \scen{II}. The accuracy of tTR, tFDA, and tQDA is reported as benchmark in the leftmost plot. Classical/robust TR and FDA classifiers, and AdaBoost, are considered under several levels of contamination.}
\label{fig:Boxplots_II}
\end{figure}

The third scenario ($\ell = \scen{III}$) describes a situation where there is no statistically significant 
difference between tQDA and tFDA, both having a higher median accuracy than AdaBoost; tTR is the method with the lowest median accuracy. 
This is confirmed by the Friedman test followed by the Wilcoxon signed-rank tests, and illustrated in Figure~\ref{fig:Boxplots_III}. The cTR method is highly affected by the presence of contamination, showing a lower accuracy, while the performance degradation of cFDA is slower with the increase of $\eps$. 
The increasing contamination level also causes an immediate increase in the angle between the subspace generated by cFDA and its theoretical counterpart. This does not hold for cTR, where the contamination causes only a little change in the angle. We find the same behavior of cTR in Figure~\ref{fig:Boxplots_III}, where even with clean data the accuracy is already quite low, and the contamination slightly worsens the performances of the classifier. Again the subspaces of rFDA and rTR are close to the theoretical ones, except for the contamination level $\eps = 0.3$.
\begin{figure} [htb!]
\centering
\includegraphics[width=\linewidth]{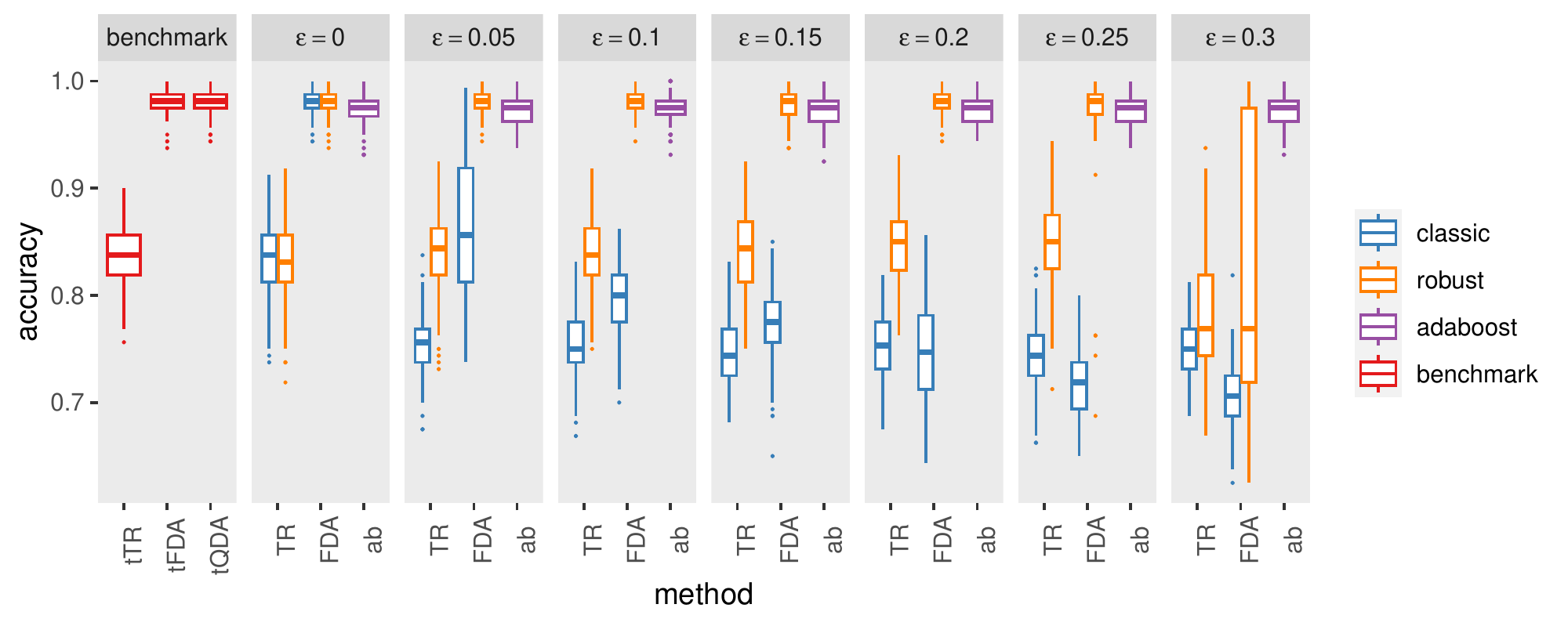}
\caption{Distribution of the accuracy over 200 simulation runs of Scenario \scen{III}. The accuracy of tTR, tFDA, and tQDA is reported as benchmark in the leftmost plot. Classical/robust TR and FDA classifiers, and AdaBoost, are considered under several levels of contamination.}
\label{fig:Boxplots_III}
\end{figure}
The fourth scenario ($\ell = \scen{IV}$) combines the distribution of scenario \scen{II} with $q\in \{0,10,20,50,100\}$ additional variables, which are irrelevant for the classification problem.
Under this scenario, in the training set we consider two group sample sizes $n_j \in \{40,\, 400\}$, clean data ($\eps=0$) and 10\% of contamination ($\eps=0.10$). We compare only cFDA, cTR, rFDA, and rTR for different numbers of irrelevant variables. The median accuracy is shown in Figure~\ref{fig:sc5_400} for $n_j = 400$ and Figure~\ref{fig:sc5_40} for $n_j = 40$. In the latter case, when $p > n_j$, the MRCD estimators are needed. 

When $n_j = 400$, Figure~\ref{fig:sc5_400} shows that the number of irrelevant variables has a small impact on the accuracy, if the data is clean ($\eps=0$). As expected, the contamination affects mainly cFDA and cTR; cFDA median accuracy seems to decrease faster than the one of cTR, but for every $q$ the median accuracy of cTR is lower than the one of cFDA.
\begin{figure}[ht!]
\centering
\includegraphics[width=.90\linewidth]{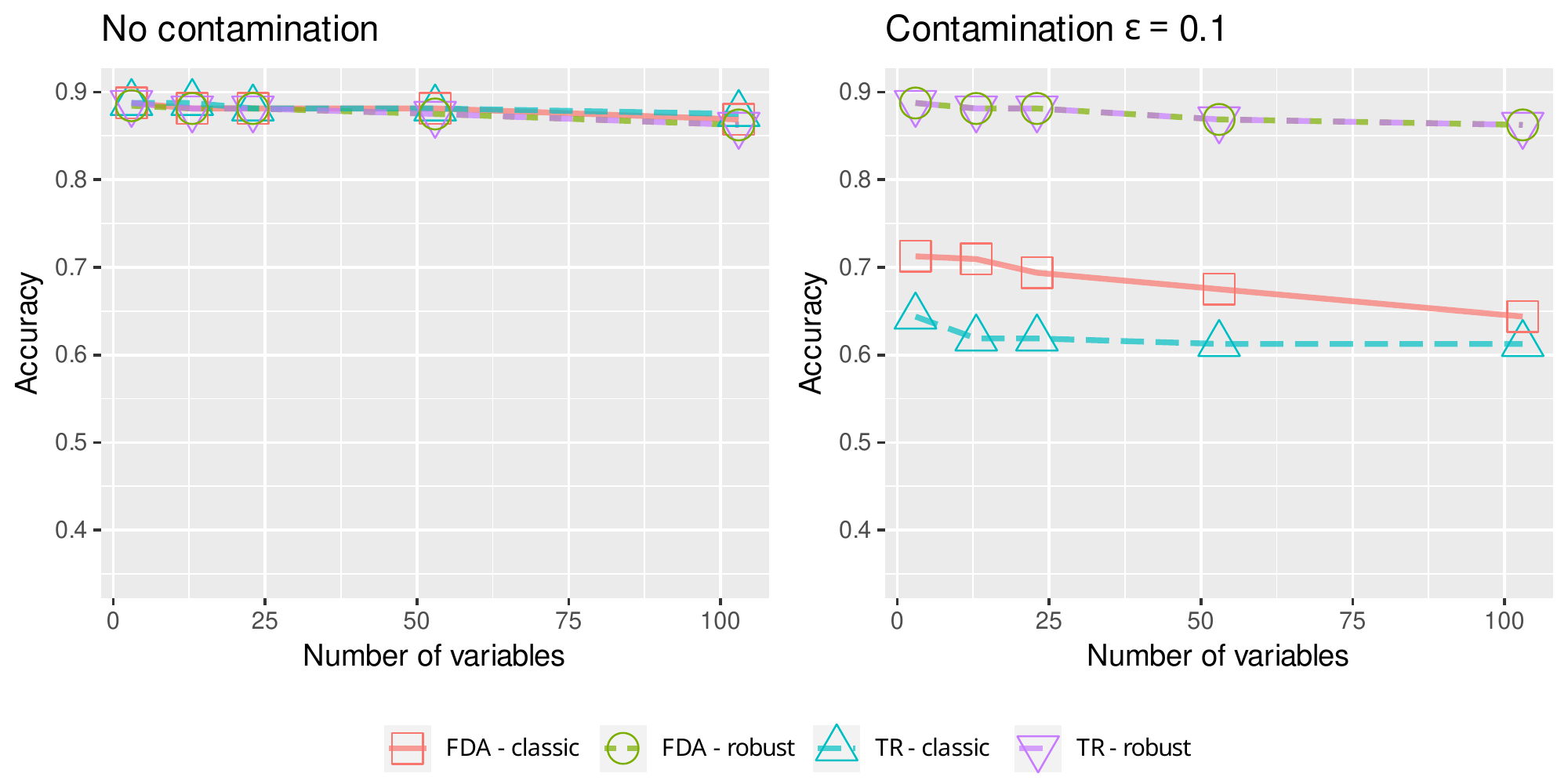} 
 \caption{Median accuracy of 200 simulations in Scenario \scen{IV}, for different numbers of irrelevant variables ($q$), with $n_j=400$.}
 \label{fig:sc5_400}
\end{figure}

\begin{figure}[ht!]
\centering
\includegraphics[width=.90\linewidth]{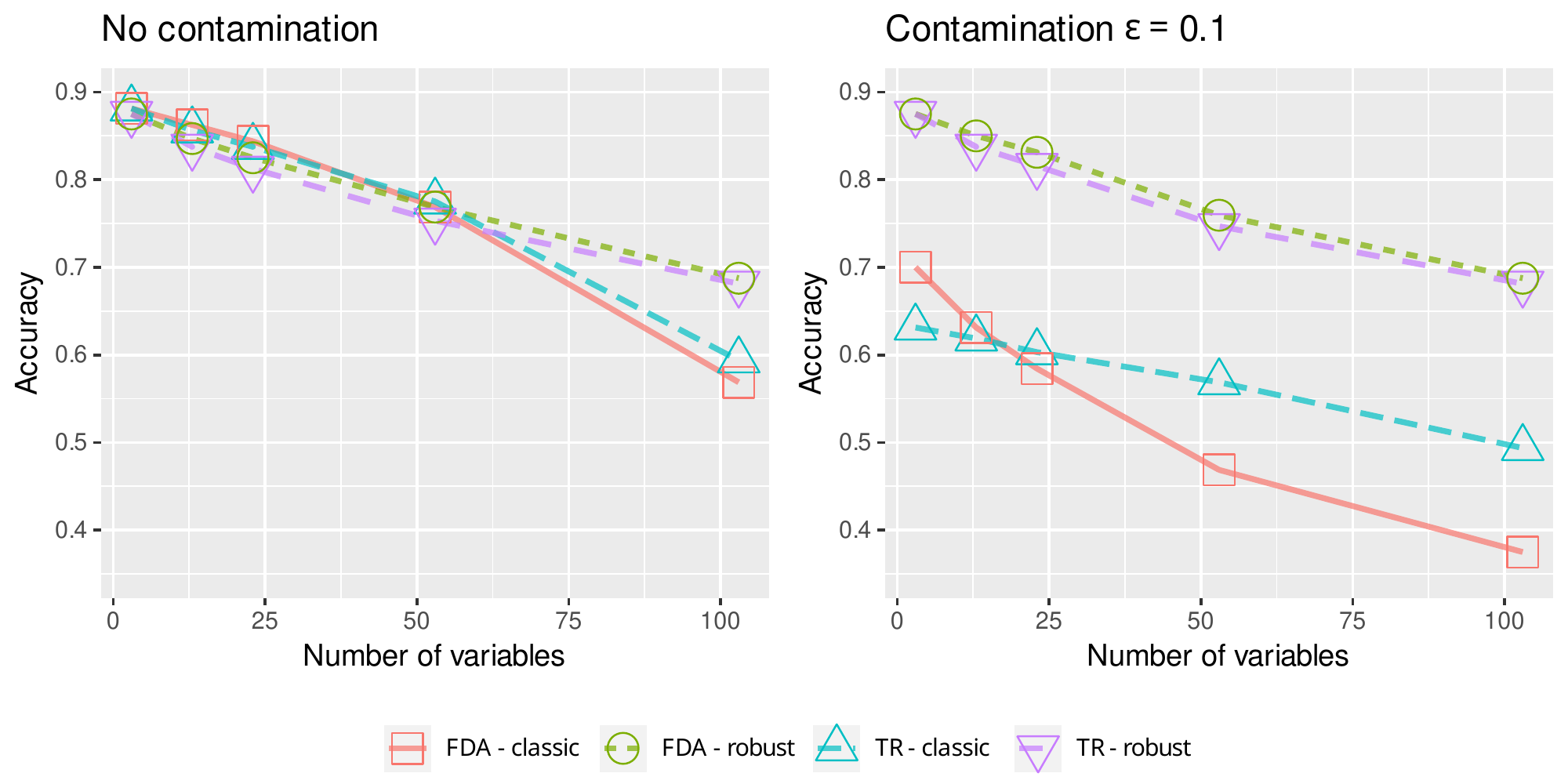} 
 \caption{Median accuracy of 200 simulations in Scenario \scen{IV}, for different numbers of irrelevant variables ($q$), with $n_j=40$.}
 \label{fig:sc5_40}
\end{figure}

The most relevant finding is the impact of the sample size. Even if both tTR and tFDA find the right theoretical matrices of coefficients (see \eqref{eq:V_scenario IV}), smaller sample sizes ($n_j = 40$) lead to poorer classification results when irrelevant variables exist, regardless the presence of contamination. In the uncontaminated case, the median accuracy for cTR (cFDA) is 0.881 (0.888) when $q = 0$, and decreases up to 0.603 (0.569) when $q=100$; in contrast, when $n_j=400$, we go from 0.881 to 0.869 for both cFDA and cTR.
When contamination is present and $n_j=40$, robust estimators achieve a similar accuracy of their classical counterparts, when no contamination is present. For classical estimators we also observe a faster decrease in median accuracy as the number of irrelevant variables increases. 

The pattern for the classical estimators under contamination is similar to the one of Figure~\ref{fig:sc5_400}, but characterized by lower median accuracies. cFDA median accuracy becomes lower than the one observed for cTR when $q\ge 25$, and decreases much faster, meaning that cFDA is clearly more sensitive to the increase of irrelevant variables than cTR. The degradation in the median accuracies of the robust estimators is quite similar to the one registered for the classical counterparts when $\eps = 0$, as expected. 

So far we have discussed three scenarios in which TR and FDA either perform equally well, or show some differences. Interestingly, scenario \scen{III} highlights the idea that a nonlinear and more complicated model is not always better than a linear one in terms of accuracy. As the contamination level increases, in scenarios \scen{II} and \scen{III} the accuracy of cFDA decreases gradually; in all three scenarios, cTR is immediately affected by the presence of contamination but its performance does not deteriorate as the contamination level increases. Even if FDA seems to be more resistant to outliers, it may end up being worst than TR for high contamination levels. Robustification helps both methods up to approximately $25\,\%$ of contamination, by construction, leading to the same accuracies of cFDA and cTR when $\eps = 0$. In scenario \scen{IV}, the increasing number of irrelevant variables affects the accuracy of both TR and FDA, regardless the chosen estimators for the covariance matrices. This behavior is more pronounced when the group sizes are small compared to the number of irrelevant variables. 

In the next section, we present FDA and TR as linear dimensionality reduction strategies. We study their performances on real datasets.

\section{Real datasets}\label{sec:realdata}
When TR and FDA are regarded as dimensionality reduction techniques, the projection step can be viewed as a pre-processing step, and a classifier is built on top of the projected data. We choose rLDA as classifier, but other alternatives (e.g., logistic regression, classification trees, etc.) are also possible. 

\begin{table}[ht]
\centering
\caption{Real datasets, with size ($n\times p$), ranges of group sizes ($n_j$), numbers of groups ($g$), and the considered maximum reduced dimension $r = \min\{p,\, g-1\}$; $p$ represents the number of variables with $Q_n > 0$.} 
\label{tab:datasets}
\begingroup\footnotesize
\begin{tabular}{lllrclc} \toprule
Dataset & Size $n\times p$ & Range $n_j$ & $g$ & $r$ & Complete name & Source \\ \midrule
{\sf ecoli} & $\ph{1}327\times\ph{1}5$ & $[\ph{1}20,\,143]$ & 5 & \ph{1}4 & Ecoli & \cite{ucirepo} \\ 
{\sf floral buds} & $\ph{1}550\times\ph{1}6$ & $[\ph{1}44,\,363]$ & 4 & \ph{1}4 & Floral Buds & \cite{R.classmap:2022} \\ 
{\sf libras} & $\ph{1}360\times90$ & $[\ph{1}24,\,\ph{1}24]$ & 15 & 14 & Movement Libras & \cite{ucirepo} \\ 
{\sf mfeat} & $2000\times64$ & $[200,\,200]$ & 10 & \ph{1}9 & Multiple Features (Karhunen-Love coefs.) & \cite{ucirepo} \\ 
{\sf penguins} & $\ph{1}333\times\ph{1}4$ & $[\ph{1}34,\,\ph{1}73]$ & 6 & \ph{1}3 & Palmer Penguins & \cite{R.palmerpenguins:2020} \\ 
{\sf plant margin} & $1600\times53$ & $[\ph{1}16,\,\ph{1}16]$ & 100 & 53 & One-hundred plant species leaves (margin) & \cite{uciplant} \\ 
{\sf plant shape} & $1600\times64$ & $[\ph{1}16,\,\ph{1}16]$ & 100 & 64 & One-hundred plant species leaves (shape) & \cite{uciplant} \\ 
{\sf plant texture} & $1599\times49$ & $[\ph{1}15,\,\ph{1}16]$ & 100 & 49 & One-hundred plant species leaves (texture) & \cite{uciplant} \\ 
{\sf segment} & $2310\times16$ & $[330,\,330]$ & 7 & \ph{1}6 & Image segmentation & \cite{ucirepo} \\ 
{\sf texture} & $5500\times40$ & $[500,\,500]$ & 11 & 10 & Texture & \cite{keelrepo} \\ 
\bottomrule
\end{tabular}
\endgroup
\end{table}

Experiments are conducted on the real datasets listed in Table~\ref{tab:datasets}. For each dataset, we discard observations containing missing values, and remove those variables for which the $Q_n$ is zero (we recall that $Q_n$ \cite{rousseeuw1993qn} is a high-breakdown point robust estimate of the scale of a random variable). This step aims at removing those variables that are almost constant on the entire dataset. For {\sf penguins}, we combine sex and species to get 6 classes, while we discard the information about the penguins' home islands. In {\sf ecoli}, we do not consider classes with less than 20 observations.

The experiments consist of two steps: first, we reduce the dimension of the problem from $n\times p$ to $n\times k$, for $k = 1, \dots, r$, via either classical or robust TR and FDA. Secondly, we use the projected data to build the robust LDA (rLDA) classifier: this is equivalent to computing the robust Mahalanobis distances between points and means in the $k$-dimensional subspace spanned by the chosen dimensionality reduction method, and assigning each data point to its closest group (we refer to Section~\ref{sec:rob} for more details). 

We compute the accuracy to evaluate the performances of rLDA combined with the various dimensionality reduction strategies. Each dataset is split into a training set, $\bX_{\rm train}$, and a test set, $\bX_{\rm test}$. The training set is used to compute the projection matrix, $\bV$, and to estimate the projected group means and pooled covariance matrix for rLDA. The test set data points are then classified according to rLDA. The median accuracy is estimated via 10-fold cross-validation. We remark that each fold contains the classes in the same proportions as the original dataset. 

When estimating the covariance matrices, we may have the following problems: the (numerical) singularity of $\bW$ estimates, and the singularity of one or more individual group covariance matrices. The first issue may be caused by possibly high correlations among the variables. The nullspace of $\bW$ may intersect the nullspace of $\bB$. In FDA (which is a generalized eigenvalue problem) this fact leads to the so-called singular pencils. In TR we may lose the uniqueness of the solution, because for some reduced dimension $k$ there may be no gap between the $k$th and the $(k+1)$st eigenvalue of $\bB - \rho^{(k)}\bW$. To avoid this situation, we compute the eigenvectors of $\bW$ and project both $\bW$ and $\bB$ onto the subspace spanned by the range of $\bW$, i.e., by the eigenvectors corresponding to its positive eigenvalues. Numerically, we select only the eigenvalues of $\bW$ satisfying $\lambda_i(\bW) > 10^{-10}\,\lambda_1(\bW)$ and the corresponding eigenspace ${\bf \Gamma}$. We solve TR and FDA for the pencil projected onto the span of the eigenvectors
$({\bf \Gamma}^T\bB{\bf \Gamma},\,{\bf \Gamma}^T\bW{\bf \Gamma})$
and thus find a projection matrix $\bV$ with fewer rows. In the end, this is equivalent to work with $\bX_{\rm train}{\bf \Gamma}$ as input of the training task; test data is projected accordingly. We adopt the same strategy for the robust covariance matrices $\wt\bB$ and $\wt\bW$.

The singularity of the individual group covariance matrices can be related to the fact that, for some $j$'s, the group size is smaller than the number of variables $n_j < p$. This is not problematic for the classical estimates $\bS_j$, and, in the worst-case scenario, it results in a singular within covariance matrix, which can be tackled as we have already suggested. As mentioned in Section~\ref{sec:rob}, in the robust case we use the MRCD estimates \cite{boudt2020mrcd} for mean and covariance, in place of the MCD estimates. These prevent the determinant of $\wt\bS_j$ from being zero.

When $n_j \ll p$, the MRCD estimates could give very different results depending on the selected subset of points. Therefore, we have repeated the same experiments with a different regularization strategy for dimensionality reduction. This consists in centering the groups with respect to their $L_1$-median (see, e.g., \cite{fritz2012comparison}) and then computing the MCD estimate of the common covariance matrix. We note that this approach implicitly assumes that all groups share the same covariance matrix. In general, we did not see any improvement compared to the use of MRCD estimates. In some cases, the performance of this other strategy is even less satisfactory.

We mention in passing that we have also compared the rLDA classifier with the classical LDA method. As expected, rLDA results in equal or a better classification performance, except for the {\sf plant shape} set. The uncommon behavior of {\sf plant shape} is also noted when using rLDA and will be discussed later.

In Figure~\ref{fig:realdata}, we report the median accuracy of each method as a function of the reduced dimension $k$. The results are also compared with rLDA where no prior dimensionality reduction is performed. The rLDA median accuracy is also computed using a 10-fold cross-validation strategy. 
\begin{figure}[htb!]
\centering
\includegraphics[width = \textwidth]{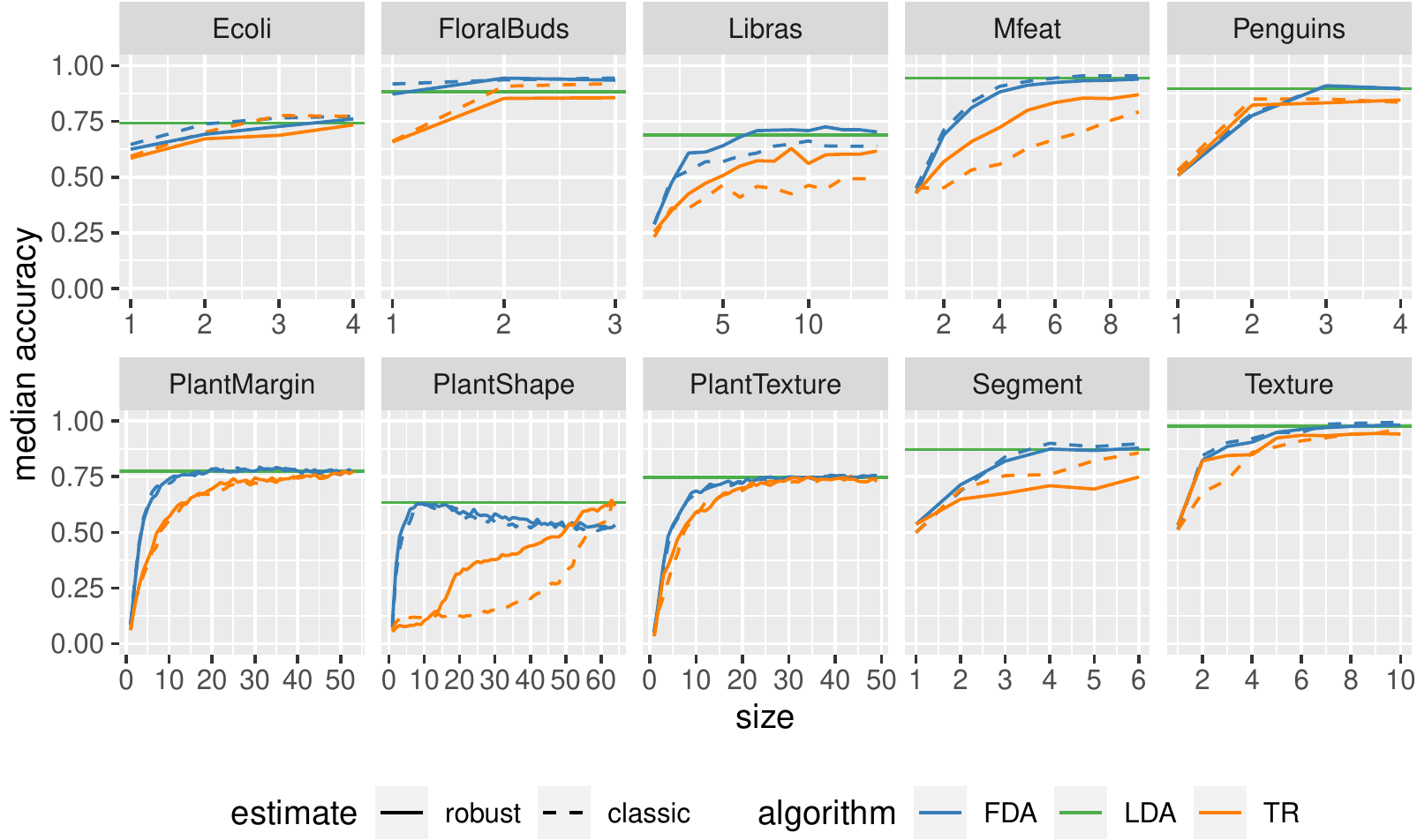}
\caption{Median accuracy for each classification problem, as a function of the reduced dimension $k$. Methods that use robust estimates (including rLDA) are indicated with solid lines; dashed lines are for classical estimates.}
\label{fig:realdata}
\end{figure}

In all problems, FDA with classical or robust estimates reaches the median accuracy of rLDA (on the original data); this means that we can reduce the dimension of the problem without losing  important information for the classification problem. In addition, the lower-dimensional representation of the data enhances interpretability and eases further possible analyses. TR does not always reach the same accuracy, and the increase in the performance looks slower than FDA. In some cases, e.g., {\sf libras}, {\sf mfeat} and {\sf segment}, it seems that a higher dimension than $k = r$ is needed to get to the same accuracy as rLDA and FDA. Finally, there are some situations where FDA and TR give a higher median accuracy than the complete model: cTR and cFDA in {\sf ecoli}; cFDA, rFDA, and cTR in {\sf floral buds}; rFDA in {\sf libras} and cFDA in {\sf segment}.

In general, FDA seems to perform as good as or better than TR, with two small exceptions: {\sf ecoli} for $k \ge 3$, {\sf penguins} for $k = 2$. In the latter case, FDA is still better when $k > 2$. While robustification seems to be of little help for FDA, with the exception of {\sf libras}, we remark that rTR can perform much better than cTR. This is the case of {\sf libras}, {\sf mfeat}, {\sf plant shape}, and {\sf Texture}.

There is an interesting and uncommon behavior in {\sf plant shape}.
For large values of $k$, the median accuracy of rFDA and cFDA decreases, while it keeps increasing for rTR and cTR. One would expect that when the reduced dimension is actually the original size of the problem, i.e., for $k = p$, rFDA coincides with rLDA. This might not be the case for two reasons: $\bVfda\in\mathbb{R}^{p\times p}$ is not an orthogonal basis for $\mathbb{R}^{p}$ and thus it does not preserve distances. Secondly, the MRCD covariance matrix is not affine equivariant \cite{boudt2020mrcd}, meaning that the estimator of the variance of the projected data is not necessarily equal to $\bVfda^T\wt\bW\bVfda$. We also point out that this is the dataset where the largest difference between FDA and TR is observed: while FDA reaches its maximum median accuracy around $k = 10$, TR needs $k = p$ to compete with rLDA.




\section{Conclusions}
\label{sec:conclusions}
We have studied and tested FDA and TR, two linear dimensionality reduction methods that are based on the within and between covariance matrices. We have shown that, under a particular contamination scenario, the two contaminated covariance matrices can be written as the sum of the clean covariance matrices, and a perturbation term that depends on the contamination level of the data. This allowed us to reinterpret the asymptotic perturbation bounds for the solution to TR \cite{zhang2013perturbation} in terms of the parameters of the original variables, the contaminating variables, and the probability of contamination. 

We have introduced a new robust version of TR, by running its algorithm with the (regularized) MCD estimates \cite{boudt2020mrcd} of the between and within covariance matrices. In the simulation study of Section \ref{sec:simulation}, the performance of the method has been compared with its classical counterpart, FDA, QDA, and AdaBoost. In all scenarios, robust estimates have a good impact on both TR and FDA performances. In some cases, FDA shows a higher accuracy than AdaBoost: this reinforces the idea that a nonlinear model does not necessarily have a higher accuracy than a linear one. We have also studied FDA and TR for an increasing number of irrelevant variables, and noticed that the accuracy of the classification decays for both FDA and TR, even when the covariance matrices are estimated robustly. This decay is faster when the group sizes are much smaller than the number of irrelevant variables.

Finally, we have assessed the performance of FDA and TR as dimensionality reduction methods on real datasets, in combination with robust LDA as classifier. In general, FDA performs as well as or better than TR; in most problems, both methods reach the performance of the full model, where the classifier is trained without any dimensionality reduction; in a few cases, they perform even slightly better. For the tested real datasets, the robust estimators have little impact on the accuracy of FDA, while, in general, a big improvement is observed in the TR method. 


\vspace{3mm}
\bmhead{Acknowledgments} This work has received funding from the European Union's Horizon 2020 research and innovation programme under the Marie Sklodowska-Curie grant agreement No 812912. 
It has also received support from Fundação para a Ciência
e Tecnologia, 
Portugal, through the project
UIDB/04621/2020. 

\bibliography{RobustRef}
\end{document}